\documentclass{article}
\usepackage[utf8]{inputenc}
\usepackage{amsthm,amsmath,amssymb,placeins}
\usepackage[usenames,dvipsnames]{xcolor}
\usepackage{tikz}
\usepackage{multicol}
\usepackage{enumerate}
\usepackage{xparse}
\usepackage{authblk}
\usepackage{soul}
\usepackage{tabularx} 

\usepackage{authblk}

\usepackage{hyperref}
\usetikzlibrary{patterns}

\newtheorem{thm}{Theorem}

\newtheorem{cor}[thm]{Corollary}

\newtheorem{lemma}[thm]{Lemma}

\newtheorem{prop}[thm]{Proposition}

\theoremstyle{definition}
\newtheorem*{defn}{Definition}
\newtheorem{quest}{Question}

\renewcommand\det{\operatorname{Det}}
\newcommand\dist{\operatorname{Dist}}
\newcommand\mbf{\mathbf}
\renewcommand\a{{\mathbf a}}
\renewcommand\b{{\mathbf b}}
\newcommand\x{{\mathbf x}}
\newcommand\V{{\mathbf v}} 
\newcommand\y{{\mathbf y}}
\newcommand\e{{\mathbf e}}
\renewcommand\c{{\mathbf c}}

\newcommand\U{{\mathbf u}}
\newcommand\0{{\mathbf 0}}
\newcommand\1{{\mathbf 1}}

\newcommand\Z{{\mathbb Z}}
\newcommand\aut{{\rm Aut}}

\renewcommand\th{^{th}}

\providecommand{\keywords}[1]
{
  \textbf{{Keywords ---}} #1
}

\providecommand{\MSC}[1]
{
  \textbf{{MSC Codes ---}} #1
}

\begin{document} 

\title{Symmetry Parameters of Various Hypercube Families}

\author[1]{Debra Boutin}
\author[2]{Sally Cockburn}
\author[3]{Lauren Keough}
\author[4]{Sarah Loeb}
\author[5]{Puck Rombach}
\affil[1]{\url{dboutin@hamilton.edu}, Hamilton College, Clinton, NY}
\affil[2]{\url{scockbur@hamilton.edu}, Hamilton College, Clinton, NY}
\affil[3]{\url{keoulaur@gvsu.edu}, Grand Valley State University, Allendale Charter Township, MI}
\affil[4]{\url{sloeb@hsc.edu}, Hampden-Sydney College, Hampden-Sydney, VA}
\affil[5]{\url{puck.rombach@uvm.edu}, University of Vermont, Burlington, VT}
\renewcommand\Affilfont{\footnotesize}

\large
\date{\today}

\maketitle

\begin{abstract}  In this paper we study the symmetry parameters determining number, distinguishing number, and cost of $2$-distinguishing, for some variations on hypercubes, namely Hamming graphs, powers of hypercubes, folded hypercubes, enhanced hypercubes, augmented hypercubes and locally twisted hypercubes. 
\par \vspace{\baselineskip}


\MSC{05C15, 05C25, 05C69}
\end{abstract}

\section{Introduction}

Hypercubes are well-studied graphs that have a number of properties that make them useful as a communications network topology, such as regularity, symmetry and  connectivity. To improve network performance, some variations on hypercubes have been proposed and studied, including Hamming graphs, powers of hypercubes,  folded hypercubes, enhanced hypercubes, augmented hypercubes and locally twisted hypercubes. In this paper, we focus on the symmetry of these graphs.  

The \emph{hypercube} $Q_n$ has vertex set $V(Q_n) = \mathbb Z_2^n = \{\V= v_1 \dots v_n \mid v_i \in \{0, 1\}\}$, with two vertices adjacent if they differ in exactly one position. Some variations on the hypercube, such as the folded hypercube $FQ_n$, enhanced hypercube $Q_{n,k}$, and augmented hypercubes $AQ_n$, have additional conditions for  two vertices being adjacent. Alternatively, $Q_n$ is the Cartesian product of $K_2$ with itself $n$ times, denoted $K_2^{\Box n}$. Hamming graphs $H(m,n)$ are a variation on the base graph, that is $H(m,n) = K_m^{\Box n}$. Formal definitions of each variation are given in their respective sections.

One way of measuring the symmetry of a graph is to examine how interchangeable its vertices are, in the following sense. An automorphism of a graph $G$ is a permutation of the vertices that preserves adjacency and non-adjacency. We let $\aut(G)$ denote the group of all automorphisms of $G$. 
A graph $G$ is {\it vertex-transitive}  if for any pair of vertices $u, v \in V(G)$, there is an automorphism $\sigma \in \aut(G)$ such that $\sigma(u) = v$. One can also define edge-transitive, arc-transitive, and distance-transitive graphs; see~\cite{Bi1993} for precise definitions. 

 Edge-transitive graphs need not be vertex-transitive, but arc-transitive graphs are always vertex-transitive and edge-transitive. Distance-transitive graphs are arc-transitive. Table~\ref{tab:trans} summarizes the types of transitivity of the various families of hypercubes that are the subject of this paper, for $n\ge 3$. The entries without citations follow from the other entries.

\begin{table}\label{tab:trans}
\begin{tabularx}{\textwidth}{|X|X|X|X|X|}
\hline
Family of hypercubes ($G$) &Vertex-Transitive &Edge-Transitive  &Arc-Transitive  &Distance-Transitive\\
\hline
\hline
Hypercubes $(Q_n)$ &Yes &Yes &Yes &Yes~\cite{Bi1993}  \\
\hline
Squares of Hypercubes $(Q_n^2)$ &Yes &Yes &Yes &Yes~\cite{Mi2021}\\
\hline
Hamming Graphs $(H_{m,n})$ &Yes &Yes &Yes &Yes~\cite{BrCoNe1989}\\
\hline
Folded Hypercubes $(FQ_n)$ &Yes &Yes &Yes 
&Yes~\cite{va2007} \\
\hline
Enhanced Hypercubes $(Q_{n,k})$ &Yes ~\cite{TzWe1991} &No ~\cite{TzWe1991} &No &No\\
\hline
Augmented Hypercubes $(AQ_n)$ &Yes~\cite{ChSu2002} &No (Prop.~\ref{prop:AugmentedTransitive}) &No (Prop.~\ref{prop:AugmentedTransitive}) &No \\
\hline
Locally Twisted Hypercubes $(LQ_n)$ &No~\cite{ChMaYa2021} &No~\cite{ChMaYa2021} &No &No\\
\hline
\end{tabularx}
\caption{Summary of 
the transitivity of various hypercube families ($n \ge 3$).}
\end{table}

Another way to measure the symmetry of a graph is to quantify what it takes to block nontrivial automorphisms. For example, we could color the vertices and require that automorphisms preserve vertex color as well as adjacency and nonadjacency. More precisely, a graph $G$ is $d$-distinguishable if there is a coloring of the vertices with colors $\{1,\ldots, d\}$ so that the only automorphism that preserves the color classes is the identity.
The {\it distinguishing number} of $G$, denoted $\dist(G)$, is the minimum $d$ such that $G$ is $d$-distinguishable. 
Another strategy is to require that automorphisms fix a subset of vertices. A set $S$  of vertices is a {\it determining set} for a graph $G$ if the only automorphism that fixes all vertices in $S$ is the identity.
The {\it determining number} of $G$, denoted $\det(G)$, is the minimum size of a determining set.
To be a distinguishing coloring, the vertices in a color class need only be fixed setwise to imply an automorphism is trivial. However, to be a determining set, the vertices must be fixed pointwise. 
One property of these parameters is that for any graph $G$,
 $\det(G) = \det(\overline G)$ and $\dist(G) = \dist(\overline G)$, where $\overline G$ denotes the complement of $G$. 
For a discussion of other elementary properties of determining numbers and distinguishing numbers, as well as the connections between them, see~\cite{AlBo2007}. 

In~\cite{BoCo2004}, the authors show that hypercubes have distinguishing number $2$ for $n\ge 4$. In fact, for a large number of graph families, all but a finite number have distinguishing number $2$. Other examples are Cartesian powers $G^{\Box n}$ of a connected graph where $G\ne K_2,K_3$ and $n\geq 2$~\cite{Al2005, ImKl2006,KlZh2007}, and Kneser graphs $K_{n:k}$ with $n\geq 6, k\geq 2$~\cite{AlBo2007}. So, in 2007, Imrich~\cite{IW} asked about refining the concept of distinguishing to provide more information about the degree of symmetry of a graph. In response, Boutin defined the \emph{cost of 2-distinguishing} a 2-distinguishable graph $G$, denoted $\rho(G)$, to be the minimum size of a color class over all 2-distinguishing colorings~\cite{Bo2008}.

In this paper we discuss the determining number, distinguishing number, and, when relevant, the cost of a number of hypercube families. Our results and our paper organization are summarized in Table~\ref{tab:OurResults}. Note that $S(r,m)$ is the number of ways to partition $[r]$ into $m$ nonempty, unlabeled parts, and throughout the paper we use $\lg n$ to mean $\log_2 n$. While the table applies for $n\ge 6$, many of the results hold for smaller $n$, and the numbers are often determined for all $n$; see the respective sections. 
 Though the distinguishing number of the general Hamming graph is unknown, much work has been done to determine under which conditions a Hamming graph is $2$-distinguishable, which allows us to compute the cost in many cases. For folded hypercubes, the determining number is one of two options, and we describe in Theorem~\ref{thm:DetFQn} the values of $n$ that give each of the options. 
 
  With the exception of Section~\ref{sec:hamming} all vertices are written as binary vectors because of the way cubes are constructed. Thus, we use bold letters to denote the vertices. That is, each vertex is denoted as $\V = v_1v_2\dots v_n$ where $v_i\in \{0,1\}$ for all $i$. Each $v_i$ is referred to as a bit, and the index $i$ is the bit's position.

\begin{table}
\begin{tabularx}{\textwidth}{|c|X|X|l|X|}
    \hline
    Section &Family of hypercubes ($G$) &$\det(G)$ &$\dist(G)$ &$\rho(G)$  \\
     \hline
     \hline
    ~\ref{sec:hypercubes} & Hypercubes ($Q_n$) &$\lceil \lg n\rceil + 1$~\cite{Bo2009a}  &2~\cite{BoCo2004} &$1+\lceil \lg n\rceil$ or $2+\lceil \lg n\rceil$~\cite{Bo2021}\\
     \hline
    ~\ref{sec:powers} &Odd Powers of Hypercubes ($Q_n^k$, $k\le n-2$) &$\lceil \lg n\rceil + 1$~\cite{Bo2009a,MiPe1994}  &2~\cite{BoCo2004, MiPe1994} &$1+\lceil \lg n\rceil$ or $2+\lceil \lg n\rceil$~\cite{Bo2021, MiPe1994}\\
     \hline
    ~\ref{sec:powers} &Even Powers of Hypercubes $(Q_n^k$, $k\le n-2$) &$\le n$ \bf{(\ref{thm:distdetpowers})} &$2$ \bf{(\ref{thm:distdetpowers})} &$\le n+1$ \bf{(\ref{thm:distdetpowers})}\\
     \hline
    ~\ref{sec:hamming} &Hamming Graphs $(H_{m,n})$ &smallest $r$ such that $n\le S(r,m)+S(r,m-1)$ \bf{(\ref{thm:DetHam})} &unknown &$\det(H(m,n))$ or $\det(H(m,n))+1$ \bf{(\ref{thm:CostHam})}\\
     \hline
   ~\ref{sec:folded} &Folded Hypercubes $(FQ_n)$ &$\lceil\lg(n+1)\rceil+1$ or $\lceil\lg(n)\rceil+2$ \bf{(\ref{thm:DetFQn})} &$2$ \bf{(\ref{thm:distFQn})} &$O(n\lg n)$ {\bf (\ref{cor:costfolded})}\\
     \hline
    ~\ref{sec:enhanced} &Enhanced Hypercubes $(Q_{n,k})$ &$\max\{\det(Q_{k-1}), $ $\det(FQ_{n-k+1}) \}$ \bf{(\ref{prop:Qnkdet})} &$2$ \bf{(\ref{thm:distenhanced})}  &unknown\\
     \hline
    ~\ref{sec:augmented} &Augmented Hypercubes $(AQ_n)$ &$2$\bf{(\ref{thm:DetAQ2})} &$2$~\cite{Cha2008} &$3$ \bf{(\ref{thm:costAQ})}\\
     \hline
    ~\ref{sec:locallytwisted} &Locally Twisted Hypercubes $(LQ_n)$ &$1$ \bf{(\ref{prop:LTQ})} &$2$ \bf{(\ref{prop:LTQ})} &$1$ \bf{(\ref{prop:LTQ})}\\
     \hline
\end{tabularx}
\caption{Summary of results for $n\ge 6$. Boldface numbers in parentheses refer to the theorem number in this paper.} \label{tab:OurResults}
\end{table}

\section{Hypercubes}\label{sec:hypercubes}

The hypercube $Q_n$ has vertex set $V(Q_n) = \mathbb Z_2^n = \{\V= v_1 \dots v_n \mid v_i \in \{0, 1\}\}$, with two vertices adjacent if they differ in exactly one position (equivalently, if they have Hamming distance 1). In~\cite{Harary00}, Harary identified the automorphism group of $Q_n$.  For any $\c \in \mathbb Z_2^n$, let $\rho_\c$ denote translation by $\c$, given by $\rho_\c(\V) = \c + \V$, where addition is modulo 2. For any $\phi$ in the symmetric group $S_n$, we let $\phi$ act on vertices by permuting positions; $\phi(v_1 \dots v_n) = v_{\phi(1)} \dots v_{\phi(n)}.$ These translations and permutations are both automorphisms of $Q_n$. Moreover, the translations constitute a normal subgroup of $\aut(Q_n)$, the only automorpism group that is both a permutation and a translation is the identity, and every automorphism can be written as a composition $\rho_\c \circ \phi$. Thus, $\aut(Q_n)$ is the semidirect product $\mathbb Z_2^n \rtimes S_n$.

The hypercube is often studied as the Cartesian product of $n$ copies of $K_2$. As shown in~\cite{Bo2009a}, a useful way to study a potential determining set of a Cartesian product is through a {\it characteristic matrix}.

Recall that a graph $G$ is said to be \emph{prime with respect to the Cartesian product} if it cannot be written as a Cartesian product of two or more graphs. Let $G=H_1\Box \cdots \Box H_n$, be a Cartesian product of graphs where each $H_i$ is prime with respect to the Cartesian product.  Each vertex of $G$ can be written uniquely as a sequence of elements of the graphs $H_1, \ldots, H_n$.  That is $\V \in V(G)$ can be written uniquely as $v_1\cdots v_n$ with each $v_i\in V(H_i)$.

\begin{defn} Let $G=H_1\Box \cdots \Box H_n$. Let $S=\{\V_1,\ldots, \V_r\}$ be an ordered subset of $V(G)$.  Define the {\em characteristic matrix} of $S$, denoted $X(S)$ or simply $X$, to be the $r\times n$ matrix whose $ij^{th}$ entry is the $j^{th}$ position of $\V_i$, denoted $v_{ij}$. Two columns of $X$ are called \emph{isomorphic} if there exists an isomorphism $\varphi:H_i \to H_j$ so that for each $k\in [r], \varphi(v_{ki})=v_{kj}$, and \emph{nonisomorphic} otherwise.\end{defn}

In~\cite{Bo2009a} Boutin proves a criterion for a set to be a determining set and uses this result to establish the determining number of $Q_n$.  These are stated as Theorems~\ref{thm:CharMatrix} and ~\ref{thm:DetQn}. In~\cite{BoCo2004}, Bogstad and Cowen prove that $\dist(Q_n)=2$ for $n\geq 4$.  This result is stated in Theorem~\ref{thm:distQn}. Finally in~\cite{Bo2021}, Boutin uses characteristic matrices again to find the cost of 2-distinguishing hypercubes. This results is stated in   Theorem~\ref{thm:costQn}.

 \begin{thm} \label{thm:CharMatrix}\cite{Bo2009a} A set of vertices $S\subseteq V(G)$ is a determining set for $G$ if and only if the set of entries in each column of the characteristic matrix $X$ contains a determining set for its appropriate factor of $G$, and  no two columns of $X$ are isomorphic.\end{thm}

\begin{thm}\cite{Bo2009a}\label{thm:DetQn} If $n\ge 1$, then $\det(Q_n)= \left\lceil {\lg n} \right\rceil+1$.  \end{thm}
 
\begin{thm}\label{thm:distQn}\cite{BoCo2004} If $n\geq 4$, then $\dist(Q_n)=2$.  Also, $\dist(Q_2)=\dist(Q_3)=3$.\end{thm}
 
\begin{thm}\label{thm:costQn}\cite{Bo2021} If $n\geq 5$, then $\rho(Q_n) \in \{1 + \lceil \lg n\rceil, 2 + \lceil \lg n\rceil\}$.  Further, there is a recurrence relation to determine which value the cost takes on for a given $n$.  Also, $\rho(Q_4)=5$. \end{thm}

\section{Powers of Hypercubes}\label{sec:powers}

Given a graph $G$, the $k\th$ power of $G$, denoted $G^k$, has vertex set $V(G)$ with an edge between vertices $u$ and $v$ if and only if $d_G(u,v)\leq k$, where $d_G(u,v)$ denotes the distance between $u$ and $v$ in $G$.

By~\cite{MiPe1994}, for all  $1\le k\le n-2$, $\aut(Q_n^k)=\aut(Q_n)$ if $k$ is odd and $\aut(Q_n^k) = \aut(Q_n^2)$ if $k$ is even. 
These automorphism groups are equal in the sense that they are the same as permutations on their (identical) vertex sets. So, the pointwise and setwise stabilizers of vertex subsets are also equal. Thus, for $1\le k\le n-2$, when $k$ is odd $Q_n^k$ and $Q_n$ have precisely the same determining sets and distinguishing colorings, and the same is true for $Q_n^k$ and $Q_n^2$ when $k$ is even. 

Thus, when $k\le n-2$ and $k$ is odd, $\det(Q^k_n) = \det(Q_n) = 1+\lceil \lg n \rceil$ by~\cite{Bo2009a}. By~\cite{BoCo2004}, $\dist(Q_2)=\dist(Q_3)=3$ and $\dist(Q_n)=2$ for $n\ge 4$. 
 For $n\geq 4$ and odd $k\le n-2$ exact values for $\rho(Q_n^k) = \rho(Q_n)$ can be found recursively; see~\cite{Bo2021}.
  
To complete our study of determining and distinguishing $Q_n^k$ we need only study $Q_n^2$. Let $h(\x,\y)$ be the distance in $Q_n$ (which is the Hamming distance), and $d(\x,\y)$ be the distance in $Q_n^2$.  Note that to find a shortest path from $\x$ to $\y$ in $Q_n^2$, it is sufficient to begin with a shortest path in $Q_n$ and replace pairs of consecutive edges in $Q_n$, say $x_{i-1} x_i x_{i+1}$, with the edge $x_{i-1} x_{i+1}$ in $Q_n^2$.  Therefore, $d(\x,\y) = \left\lceil \frac{h(\x,\y)}{2}\right\rceil$. 

Lemma~\ref{lem:inducedsg} applies to all graphs, but we will use it specifically in the proofs of Theorem~\ref{thm:distdetpowers} and Theorem~\ref{thm:distFQn}.

\begin{lemma} \label{lem:inducedsg}  Let $T$ be a determining set for graph $G$.  If the subgraph induced by $T$, $G[T]$ is asymmetric then  $G$ is $2$-distinguishable and $T$ is a color class in a $2$-distinguishing coloring. \end{lemma}

\begin{proof} Suppose we have a determining set $T$ for $G$ with the property that $G[T]$ is asymmetric. Suppose $\varphi \in \aut(G)$ preserves $T$ setwise. It is immediate that restricting $\varphi$ to $G[T]$ is an automorphism. Since $G[T]$ is asymmetric, the restricted $\varphi$ acts trivially on $T$. 
However, since $T$ is a determining set, any automorphism that fixes $T$ pointwise is trivial. We conclude that $\varphi$ is the identity. Thus, $T$ is a color class in a $2$-distinguishing coloring and therefore $G$ is $2$-distinguishable. \end{proof}

\begin{defn}  Let $\U_0 = \0$.  For $i\in[n]$ define $\U_i$ to have 1's in positions $1,\ldots, i$ and 0's elsewhere. \end{defn}

\noindent{\bf Observations:}

\begin{enumerate}  

    \item \label{obs:DifferInIth} For $1\le i \le n$, $\U_{i-1}$ and $\U_i$ differ only by the 1 in the $i\th$ position of $\U_i$. So,
    
        \begin{enumerate}
    
        \item \label{obs:HammingDiff} for any $\x=x_1\ldots x_n\in V(Q_n^2)$, $h(\U_i,\x)- h(\U_{i-1}, \x)\in \{\pm 1\}$
    
        \item \label{obs:ParityAlt} the parity of $h(\U_i,\x)$ is distinct from  the parity of $h(\U_{i-1},\x)$ 
    
        \item \label{obs:FindXi} $x_i=1$ if and only if $h(\U_i,\x)<h(\U_{i-1},\x)$ 
        
        \item \label{obs:GraphDistXi} $d(\U_i,\x) < d(\U_{i-1},\x)$ if and only if $x_i=1$.
    
        \end{enumerate}
        
    \item \label{obs:GraphDist} For any two vertices $\x$ and $\y$, $h(\x,\y)$ is even if and only if $h(\x,\y)= 2d(\x,\y)$, while $h(\x,\y)$ is odd  if and only if $h(\x,\y)=2d(\x,\y)-1$.
   
    \end{enumerate}

\begin{thm}\label{thm:distdetpowers} For $n>3$, $\det(Q_n^2) \leq n$, $\dist(Q_n^2)=2$, and $\rho(Q_n^2)\leq n+1$. \end{thm}

\begin{proof} Let $S=\{\U_0, \U_1,\ldots, \U_{n-1}\}$.  Since the induced subgraph of $S$ in $Q_n$ is isomorphic to $P_n$, the path on $n$ vertices, the induced subgraph of $S$ in $Q_n^2$ is isomorphic to $P_{n}^2$. If we are given only the vertices of $S$ and not their corresponding binary sequences, we can still infer their relative order from their induced subgraph, and without loss of generality, up to isomorphic image, which vertex is which. So we can assume we are given $S$ as the ordered set of vertices of $S$, in order, and can compute the distance from any $\x\in V(Q_n^2)$ to each of the $\U_i$. 

Let $\x\in Q_n^2$.

{\bf Case 1)} Suppose there exists $i\in[2,n-1]$ so that $d(\U_{i-1},\x)\ne d(\U_i,\x)$. 

As shown in Observation~\ref{obs:HammingDiff}, $h(\U_i,\x) - h(\U_{i-1},\x) \in \{\pm 1\}$.  So in general, $d(\U_{i}, \x) - d(\U_{i-1},\x)\in\{-1,0,1\}$.  In this case, since $d(\U_{i-1},\x)\ne d(\U_{i},\x)$, we have $d(\U_{i}, \x) - d(\U_{i-1},\x)\in\{-1,1\}$. Using that $d(\x,\y) = \left\lceil \frac{h(\x,\y)}{2}\right\rceil$, we can conclude that
$d(\U_{i-1},\x)< d(\U_{i},\x)$ if and only if $h(\U_{i-1},\x)$ is even, and
$d(\U_{i-1},\x) > d(\U_{i},\x)$ if and only if $h(\U_{i-1},\x)$ is odd.

Since we now know the parity of $h(\U_{i-1},\x)$, by Observations~\ref{obs:ParityAlt} and~\ref{obs:GraphDist} we can find the parity and values for each $h(\U_i,\x)$ for $i\in[0,n-1]$. Then we can use Observation~\ref{obs:FindXi} to find the first $n-1$ positions of $\x$.  Further, since we know the parity of $h(\U_0,\x)$, we can use our knowledge of $x_1, \dots, x_{n-1}$, to find $x_n$.

{\bf Case 2)} Suppose that for all $i\in[2,n-1]$, $d(\U_{i-1},\x) = d(\U_i,\x)$.

Since the parity of $h(\U_i,\x)$ alternates, but $d(\U_i,\x)$ remains the same, the Hamming distance alternates between an even value and the odd value that is one less.  Recall that by Observation~\ref{obs:FindXi}, $x_i=1$ if and only if $h(\U_i,\x)<h(\U_{i-1},\x)$. Thus, the only possibilities for  $\x$ are $\a=0101\cdots a_n \text{ or } \b=1010\cdots b_n$.

If $\x=\a$, the first positon is $a_1=0$, and so by Observation~\ref{obs:FindXi}, $h(\U_0,\a) < h(\U_1,\a)$. Further, the fact that $d(\U_0,\a) = d(\U_1,\a)$ tells us that $h(\U_0,\a)$ is odd. Together with $d(\U_0,\a)$ we can find the Hamming weight of $\a$ and use our knowledge of $a_1, \dots, a_{n-1}$ to compute $a_n$.

Similarly, if $\x=\b$,  the first position is $b_1=1$, and so by Observation~\ref{obs:FindXi}, $h(\U_0,\b) > h(\U_1,\b)$ and we use the fact that $d(\U_0,\b) = d(\U_1,\b)$ to conclude that $h(\U_0,\b)$ is even. Together with $d(\U_0,\b)$ we can find the Hamming weight of $\b$ and use our knowledge of $b_1,\dots, b_{n-1}$, we compute $b_n$.

Since each vertex of $Q_n^2$ can be uniquely determined by its relationship to the vertices of $S$, $S$ is a determining set for $Q_n^2$.

Let $T=S\cup \{\bf w\}$ where $\bf w$ has a 0 in position 1 and 1's everywhere else. Clearly, $\bf w$ differs in at least three positions from $\U_0$. Furthermore, $\bf w$ differs in the first and last position from every other vertex in the set $S$. With the exception of $\U_{n-1}$, it differs in at least one more position from every vertex in $S$. Therefore, $G[T]$ is isomorphic to $P_{n}^2$ with a pendant edge attached to one of the two vertices of degree 2. Thus, $G[T]$ is asymmetric and by Lemma~\ref{lem:inducedsg}, $T$ is a color class in a $2$-distinguishing coloring for $Q_n^2$ and $\dist(Q_n^2)=2$.
\end{proof}

When $k=n-1$, $Q_n^k$ is the complement of a matching, and when $k\ge n$, $Q_n^k = K_n$. Theorem~\ref{thm:distdetpowers} and results from~\cite{MiPe1994,Bo2021,BoCo2004} lead to Corollary~\ref{cor:powers}.
\begin{cor}\label{cor:powers}
Let $n\ge 4$ and $1\le k\le n-2$. For $k$ odd, \[\det(Q_n^k)=1+\lceil \lg n\rceil,\,\, \dist(Q_n^k) = 2,\] and $\rho(Q_n^k)$ can be found recursively.  For $k$ even, 
\[\det(Q_n^k) \le n,\, \, \dist(Q_n^k) = 2,\,\, \text{ and } \,\, \rho(Q_n^k)\le n+1.\]
\end{cor}

 \section{Hamming Graphs}\label{sec:hamming}
 
Hamming graphs are a natural generalization of hypercubes. While $Q_n$ is the Cartesian product of $n$ copies of $K_2$, the Hamming graph $H(n,m)$ is the Cartesian product of $n$ copies of $K_m$.  Thus, $K_m^{\Box n}$ is an alternate notation for a Hamming graph.

Because Hamming graphs are a Cartesian product of graphs, characteristic matrices are a useful tool for studying symmetry parameters of Hamming graphs. The result in Theorem~\ref{thm:DetHam} comes directly from Theorem 5 in~\cite{Bo2009a}.  The idea behind the proof is similar, but here we more directly find a value for the determining number using Stirling numbers of the second kind. Recall that $S(r,m)$, a Stirling number of the second kind, counts the number of partitions of $[r]$ into $m$ nonempty, unlabeled parts. See Section 3.3 in~\cite{We2020} for more information on Stirling numbers of the second kind.  

\begin{thm}\label{thm:DetHam} The determining number of $H(n,m)$ is the smallest integer $r$ for which $n\leq S(r,m)+S(r,m-1).$ \end{thm}

\begin{proof} By Theorem~\ref{thm:CharMatrix},  $S\subset V(H(m,n))$ of cardinality $r$ is a determining set if and only if its $r \times n$ characteristic matrix contains $n$ nonisomorphic columns, each of which has at least $m-1$ distinct vertices of $K_m$ among its entries.  We will first count the number of nonisomorphic columns that contain all $m$ vertices of $K_m$, and then count all those that contain precisely $m-1$ such vertices. Note that in order for a column of length $r$ to contain $m$ distinct vertices of $K_m$, we require that $r\geq m$.  In particular this tells us that $\det(H(m,n))\geq m$.

There is a bijection between partitions of $[r]$ into $m$ nonempty labeled parts  and  columns of length $r$ containing $m$ distinct vertices of $K_m$.  This bijection is achieved by placing $i\in [r]$  in part $j$ of the partition  precisely when the column has vertex $j\in [m]$ in its $i\th$ position.

Two columns of length $r$ are isomorphic if and only if their associated partitions are the same up to a permutation of the part labels.  Thus, to count the number of nonisomorphic columns containing at least one of each of the vertices of $K_m$, we need only count the number of ways to partition $[r]$ into $m$ nonempty unlabeled sets. This is precisely $S(r,m)$.

Next we want to count the number of nonisomorphic columns of length $r$ with  precisely $m-1$ distinct vertices of $K_m$.  Note that if we use the specific vertices $1,\ldots m-1$, then we can compute this as $S(r,m-1)$.  Further, any column of length $r$ with precisely $m-1$ of the vertices of $K_m$ is isomorphic to a column containing the specific vertices $\{1,\ldots, m-1\}$.  Thus, the  number of nonisomorphic columns of length $r$ with  precisely $m-1$ distinct vertices of $K_m$ is $S(r,m-1)$.

Thus, there are at most $S(r,m)+S(r,m-1)$ nonisomorphic columns in a $r\times n$ characteristic matrix for a determining set of $H(m,n)$.
Thus, if $n\leq S(r,m)+S(r,m-1)$ we can create a $r\times n$ characteristic matrix for a determining set of size $r$ for $H(m,n)$.   Thus, $\det(H(m,n))\leq r$. Further, if $n<S(r-1,m) + S(r-1, m-1)$, we could create an $(r-1)\times n$ characteristic matrix of a determining set of $H(m,n)$
implying that 
$\det(H(m,n)) < r$.
Thus, if $r$ is the smallest integer for which $n\leq S(r,m)+S(r,m-1)$, then $\det(H(m,n))=r$.
\end{proof}

To achieve a closed formula for $\det(H(m,n))$ we can start with the formula from~\cite{We2020} page 121, and adjust by dropping the final term of 0, yielding \[S(r,m)=\frac{1}{m!} \sum_{i=0}^{m-1} (-1)^i {m\choose i} (m-i)^r.\]  With a little bit of algebra, $S(r,m) + S(r, m-1)$ becomes

\[\frac{1}{m!}\sum_{i=0}^{m-1} (-1)^i {m\choose i} (m-i)^r + \frac{1}{(m-1)!}\sum_{i=0}^{m-2} (-1)^i {m-1\choose i} (m-i-1)^r\]

\[=\frac{1}{(m-1)!} \left( (-1)^{m-1}  + \sum_{i=0}^{m-2} (-1)^i {m-1 \choose i} \Big( (m-i)^{r-1} + (m-i-1)^r \Big) \right).\]

\begin{cor} $\det(H(m,n))$ is the smallest integer $r$ for which 
\[n\leq \frac{1}{(m-1)!} \left( (-1)^{m-1}  + \sum_{i=0}^{m-2} (-1)^i {m-1 \choose i} \Big( (m-i)^{r-1} + (m-i-1)^r \Big) \right).\] \end{cor}

A number of mathematicians have studied the distinguishing number of Hamming graphs. The following theorem summarizes their results.

\begin{thm}\label{thm:DistHam}\cite{Al2005,BoCo2004, ImKl2006,KlZh2007} The Hamming graph $H(m,n)$ is 2-distinguishable if and only if \[(i)\,  m=2 \text{ and } n\geq 4, \text{ or } (ii) \, m=3 \text{ and } n\geq 3, \text{ or } (iii)\,  m\geq 4 \text{ and }  n\geq 2 .\]\end{thm}

In~\cite{Bo2013a}, Boutin studied the cost of distinguishing Cartesian powers of prime, connected graphs, as stated below.

\begin{thm}\label{thm:CostCart}\cite{Bo2013a} Let $G$ be a prime, connected graph on at least three vertices with $\det(G) \le n$.
If $G^{\Box n}$ is a $2$-distinguishable and $\max\{2, \det(G)\}$ $<$ $\det(G^{\Box n})$, then \[\rho(G^{\Box n}) \in \{\det(G^{\Box n}), \det(G^{\Box n})+1\}.\]\end{thm}

Now, let $G = K_m$ with $m \ge 3$. Note that cliques are prime with respect to the Cartesian product.  
Since $\det(K_m)=m-1$ and we  showed in the proof of Theorem~\ref{thm:DetHam} that $\det(H(m,n))\geq m$, we have $\det(K_m)< \det(K_m^{\Box n}) = \det(H(m,n))$.  Thus, we can use Theorem~\ref{thm:CostCart} to achieve the following theorem.

\begin{thm}\label{thm:CostHam}   If $2 \le m-1 \leq n$, and $\dist((H(m,n))=2$, then \[\rho(H(m,n)) \in \{\det(H(m,n)), \det(H(m,n))+1\}.\]\end{thm}

\section{Folded Hypercubes}\label{sec:folded}

The folded hypercube graph $FQ_n$ is the result of adding edges between opposite vertices in the hypercube $Q_n$. More precisely, two vertices are adjacent if and only if their Hamming distance is either $1$ or $n$. These graphs were introduced by El-Amaway and Latifi in 1991~\cite{ElLa1991}, who showed that they had better network communication parameters than regular hypercubes.  

Note that $FQ_1 = K_2$, $FQ_2 = K_4$ and $FQ_3 = K_{4,4}$,  where one part consists of all vertices with odd Hamming weight and the other consists of all vertices with even Hamming weight. The determining and distinguishing numbers of these graphs are known: $\det(FQ_1) = 1$, $\det(FQ_2) = 3$,
and $\det(FQ_3) = 6$,  and $\dist(FQ_1)= 2$, $\dist(FQ_3) = 3$, and $\dist(FQ_3) = 5.$

Recall from Section~\ref{sec:hypercubes}, $\aut(Q_n)$ is the semidirect product of translations and $S_n$, where the elements of $S_n$ simply permute the $n$ bits. Mirafzal showed in~\cite{Mirafzal2016SomeOA} that $\aut(FQ_n)$ is closely related to  $\aut(Q_n)$. When $n\ge 4$, $\aut(FQ_n)$ is the semidirect product of translations and $S_{n+1}.$ More precisely,  any automorphism of $FQ_n$ is of the form $\rho_{\bf c} \circ \phi$ where $\rho_{\bf c}$ is translation by $\bf c$ and $\phi$ is a linear extension of a permutation of $\{\e_1, \dots , \e_n, \bf 1\}$ where $\e_i$ is the standard basis vector with a 1 in position $i$ and 0's elsewhere. That is,
\[
\phi(\a) = \phi(a_1a_2\dots a_n) = a_1\phi(\e_1)  + \dots + a_n\phi(\e_n).
\]
If $\phi(\e_i) = \1$ and $\phi(\1) = \bf e_\ell$, then $\phi(\a) = \phi(a_1a_2\dots a_n)$ is 
\begin{equation}\label{eq:phiAction}
 (a_i+a_1)\phi(\e_1) +  \dots + (a_i + a_{i-1})\phi(\e_{i-1}) + (a_i + a_{i+1})\phi(\e_{i+1})  + \dots (a_i+a_n)\phi(\e_n) + a_i\e_\ell. 
\end{equation}
Equivalently, given the conditions above, if $\phi(\e_j) = \e_k$, then the $k^{th}$ position of $\phi(\a)$ is $a_i+a_j$. Thus,  $\phi$ fixes $\a$ if and only if $a_i = a_\ell$ and $a_i+a_j = a_k$ whenever $\phi(\e_j)=\e_k$. 
Note that if $\phi \in S_{n+1}$ fixes $\1$, then $\phi$ simply permutes the bits. Thus, $\aut(Q_n)$ is a subgroup of $\aut(FQ_n)$.

The lemma below gives some necessary conditions for $S  \subseteq V(FQ_n)$ to be a determining set for $FQ_n$. 

\begin{lemma}\label{lem:DetFactsFolded}
Let $S  \subseteq V(FQ_n)$ be a determining set for $FQ_n$, where $n \ge 4$. Then
\begin{enumerate}
    \item\label{part:DetFactsFolded1}  $S$ is a determining set for $Q_n$,
    \item \label{part:DetFactsFolded2}there is a determining set for $FQ_n$ of size $|S|$ containing $\0$, 
    \item \label{part:DetFactsFolded3}for each $i \in [n]$, there is some ${\bf v} \in S$ such that $v_i = 1$. 
\end{enumerate}
\end{lemma}

\begin{proof} 
 For part~\ref{part:DetFactsFolded1}, note that $S$ is also a subset of $V(Q_n)$. If $S$ is not a determining set for $Q_n$, then there exists a nontrivial automorphism $\sigma \in \aut(Q_n)$ that fixes each element of $S$. Since $\aut(Q_n)$ is a subgroup of $\aut(FQ_n)$, $S$ is not a determining set of $FQ_n$. 
 
 For part~\ref{part:DetFactsFolded2}, let ${\bf a} \in S$ and let $S' = \rho_{\bf a}(S) = \{ {\bf a + v} \mid {\bf v} \in S\}.$ By part~\ref{part:DetFactsFolded1}, $S$ is a determining set for $Q_n$. So, by Theorem~\ref{thm:CharMatrix}, the columns of the characteristic matrix $X(S)$ are nonisomorphic.  Adding $\bf a$ to each vertex in $S$ has the effect of switching all entries in the $i^{th}$ column from 0's to 1's and 1's to 0's. 
 This process cannot make two nonisomorphic columns isomorphic, so $X(S')$ still satisfies  Theorem~\ref{thm:CharMatrix}, which means $S'$ is still a determining set for $Q_n$. Moreover, $\a+\a=\0 \in S'$.
 
 Suppose $\sigma = \rho_{\bf c } \circ \phi \in \aut(FQ_n)$ fixes every vertex in $S'$; we will show $\sigma$ must be the identity. Every $\phi \in S_{n+1}$ fixes $\0$  and so $\sigma = \rho_{\bf c} \circ \phi$ fixes $\bf 0$ if and only if $\bf c = \0$. Thus, $\sigma = \phi$. If $\phi$ fixes every vertex in $S'$ then for any ${\bf v} \in S$, 
 \[
 {\bf a+ v} = \phi({\bf a + v}) = \phi(\bf a) + \phi(\bf v). 
 \]
 since $\phi$ acts linearly. This implies 
 \[
 \rho_{{\bf a} + \phi({\bf a)}}\circ \phi(\bf v) =  ({\bf a} + \phi({\bf a)}) + \phi({\bf v}) = \bf a + \bf a + \bf v = \bf v.
 \]
 Thus, $\rho_{{\bf a} + \phi({\bf a)}}\circ \phi$ fixes every vertex in $S$. By the assumption that $S$ is a determining set for $FQ_n$, we conclude that $\rho_{{\bf a} + \phi({\bf a)}}\circ \phi$ is the identity, which implies that  $\phi$ is the identity in $S_{n+1}$.

 For~\ref{part:DetFactsFolded3}, suppose that every vertex in $S$ has a $0$ as in position $i$.  Let $\phi \in S_{n+1}$ be the permutation that transposes $\bf e_i$ and $\bf 1$, and for all $j \neq i$, $\phi(\bf e_j) = e_j.$ By equation (\ref{eq:phiAction}), $\phi$ is a nontrivial automorphism that fixes any vertex with a $0$ in position $i$, and hence $\phi$ fixes every vertex in $S$. This contradicts the assumption that $S$ is a determining set.
 \end{proof}
 
A consequence of part~\ref{part:DetFactsFolded2} of Lemma~\ref{lem:DetFactsFolded} is that when investigating potential determining sets for folded hypercubes, we can restrict ourselves to sets containing $\0$. Recall from the proof of Lemma~\ref{lem:DetFactsFolded} that $\sigma = \rho_{\bf c} \circ \phi$ fixes $\0$ if and only if $\bf c = \0$. Hence, the only automorphisms that can fix all vertices in a set containing $\0$ are elements of $S_{n+1}$.

Because any set of vertices $S$ of $FQ_n$ is also a set of vertices of $Q_n$ by part~\ref{part:DetFactsFolded1} of Lemma~\ref{lem:DetFactsFolded}, we can form its characteristic matrix $X(S)$. The next few results relate properties of this matrix to whether $S$ is a determining set for $FQ_n$.
 
 \begin{cor}\label{cor:NecessaryXS}
 For $n\ge 4$, if $S$ is a determining set of $FQ_n$ containing $\0$, then its characteristic matrix $X(S)$ has a zero row and distinct, nonzero columns.
 \end{cor}
 
 \begin{proof}
  Since $S$ contains $\0$, $X(S)$ must contain a zero row. Since $S$ is a determining set for $Q_n$, by Theorem~\ref{thm:CharMatrix}, the columns must be nonisomorphic, and therefore distinct. Finally, part~\ref{part:DetFactsFolded3} of the Lemma~\ref{lem:DetFactsFolded} implies that $X(S)$ cannot have a zero column.
 \end{proof}

\begin{prop}\label{prop:DetFQbounds}
For all $n \ge 4$,
$
\lceil\lg(n+1)\rceil + 1 \le \det(FQ_n)\le \lceil \lg(n)\rceil +2.
$
\end{prop}

\begin{proof}
To prove the upper bound, let $S_Q$ be a minimum determining set for $Q_n$ containing $\0$. By Theorem~\ref{thm:DetQn}, $|S_Q|= \lceil \lg(n)\rceil +1$. The only automorphisms of $Q_n$ that fix $\0$ are permutations of the bits; these are also the only automorphisms that fix $\1$. Hence, by minimality, $\1 \notin S$.  Let $S = S_Q \cup \{ \1 \}$, so that $|S|= \lceil \lg(n)\rceil +2$. It suffices to show that  $S$ is a determining set for $FQ_n$.

Since $S_Q$ is a determining set for $Q_n$ containing $\0$, the only nontrivial automorphism of $FQ_n$ that could fix every vertex in $S$ is of the form $\phi \in S_{n+1}$ where $\phi(\bf e_i) = \1$ for some $i \in [n]$. Since such a $\phi$ does not fix $\1$, it does not fix every vertex in $S$.

Next we prove the lower bound.  
Let $T$ be a determining set for $FQ_n$ containing $\0$ with characteristic matrix $X(T)$. 
By Corollary~\ref{cor:NecessaryXS}, the  characteristic matrix $X(T)$ has a zero row and distinct,  nonzero, columns. This occurs if and only if the columns in the submatrix of $X(T)$ obtained by removing the zero row are distinct, nonzero binary columns of length $|T|-1$.
Since there are only $2^{|T|-1} - 1$ distinct nonzero binary columns of length $|T|-1$, 
$n \le 2^{|T|-1} - 1$, which implies $\lceil \lg(n+1)\rceil +1 \le |T|$. 
\end{proof}

Lemma~\ref{lem:Fixins} will allow us to show that both bounds in  Theorem~\ref{prop:DetFQbounds} are sharp. In what follows, ${\bf col}_i M$ denotes the $i^{th}$ column of the matrix $M$.

\begin{lemma}\label{lem:Fixins}
Let $S \subset V(FQ_n)$, $n \ge 4$, 
and assume that $X(S)$ has a zero row and distinct, nonzero columns.
Suppose $\phi \in S_{n+1}$ satisfies $\phi(\e_i) = 1$.  If $\phi$ fixes every vertex in $S$, then $\phi$ is a product of transpositions that leaves no $\e_j$ fixed, $n$ is odd,
and the sum of the columns of $X(S)$ is
\[
\sum_{s=1}^n {\bf col}_s X(S) = \begin{cases} {\bf col}_i X(S), \quad &\mbox{ if }  n \equiv 1 \pmod 4\\
\0,  &\mbox{ if } n \equiv 3 \pmod 4.
\end{cases}
\]
\end{lemma}

\begin{proof}
Assume that $\phi$ fixes every vertex in $S= \{ \a_1, \dots, \a_r\}$, 
 where $\a_t =a_{t1} a_{t2} \dots a_{tn}$. For $j \neq i$, suppose $\phi(\e_j) = \e_k$.
 Since $\phi(\a_t) = \a_t$, the $k^{th}$ position of $\phi(\a_t)$ is  $a_{tk}$. On the other hand, by equation~(\ref{eq:phiAction}),  the $k^{th}$ position of $\phi(\a_t)$ is $a_{ti}+ a_{tj}.$ Hence, for all $t$, $a_{ti}+ a_{tj} = a_{tk}$, which equivalent to $a_{ti}= a_{tj} + a_{tk}$. Thus, ${\bf col}_i X(S) = {\bf col}_j X(S) + {\bf col}_k X(S)$. Since the columns of $X(S)$ are nonzero, $j \neq k$, so $\phi$ does not fix any $\e_j$.

Suppose $\phi(\e_k) = \e_m.$ Then by the same reasoning, ${\bf col}_i X(S) = {\bf col}_k X(S) + {\bf col}_m X(S)$. If $m \neq j$, this contradicts of the assumption that $X(S)$ has distinct columns. Hence, $\phi$ must transpose $ \e_j \leftrightarrow \e_k$.

Finally, assume $\phi(\1) = \e_\ell.$ Then by equation (\ref{eq:phiAction}) and the assumption that $\phi(\a_t) = \a_t$, $a_{ti} = a_{t\ell}$ which implies ${\bf col}_i X(S) = {\bf col}_\ell X(S)$. By the assumption that the columns of $X(S)$ are distinct, $i = \ell$, which means that $\phi$ transposes $\1 \leftrightarrow \e_i$.

Thus, as a permutation, $\phi$ is a product of transpositions that leaves no $\e_j$ fixed. In particular, this implies that $n+1$ is even and so $n$ must be odd. For  each transposition $\e_j \leftrightarrow \e_k$, the matrix $X(S)$ satisfies 
$
{\bf col}_i X(S) = {\bf col}_j X(S) + {\bf col}_k X(S).
$
This implies that the sum of the columns of $X(S)$ is
\[
\left(1 + \frac{n-1}{2}\right ) {\bf col}_i X(S) = \left(\frac{n+1}{2}\right) {\bf col}_i X(S) = \begin{cases} {\bf col}_i X(S), \quad &\mbox{ if }  n \equiv 1 \pmod 4\\
\0,  &\mbox{ if } n \equiv 3 \pmod 4.
\end{cases}
\]
\end{proof}

\begin{cor}\label{cor:detSetS}
Let $S$ be a set of vertices of $FQ_n$, $n \ge 4$,  whose characteristic matrix $X(S)$ has a zero row and distinct, nonzero columns. If either
\begin{enumerate}
    \item \label{part:detSetS1} $n$ is even, or
    \item \label{part:detSetS2} $n\equiv 1 \pmod 4$ and the column sum of $X(S)$ is not one of the columns of $X(S)$, or
    \item \label{part:detSetS3} $n\equiv 3 \pmod 4$ and the column sum of $X(S)$ is not $\0$,
\end{enumerate} 
then $S$ is a determining set for $FQ_n$.
\end{cor}

\begin{proof}
By Theorem~\ref{thm:CharMatrix}, $S$ is a determining set for $Q_n$. Hence, if $S$ is not a determining set for $FQ_n$, then there must exist $\phi \in S_{n+1}$, with $\phi(\e_i) = 1$ for some $i \in [n]$ that fixes every vertex in $S$. Now apply Lemma~\ref{lem:Fixins}.
\end{proof}

\begin{thm}\label{thm:DetFQn}
 The determining number of the folded cube is
 \[
 \det(FQ_n) = \begin{cases}
 1, \quad \quad &\mbox{ if } n=1,\\
 3, &\mbox{ if } n=2,\\
 6, &\mbox{ if } n=3,\\
 \lceil \lg(n)\rceil + 2, &\mbox{ if } n = 2^m-1 \text{ or } 2^m-3 \text{ for } m\ge 3,\\
 \lceil \lg(n+1)\rceil + 1, & \text{ otherwise.}
 \end{cases}
 \]
\end{thm}

 \begin{proof} The cases where $n\le 3$ were discussed at the beginning of this section, so in the remainder of this proof we assume $n \ge 4.$

First we consider the case where $n$ is even. For simplicity, let $m = \lceil \lg(n+1)\rceil$. Then we define a binary $(m+1) \times n$ matrix by making the top row all 0's and filling in the remaining entries by choosing $n$ distinct nonzero binary columns of length $m$. Since $n \le 2^m-1$, this is possible.  We interpret $X$ as the characteristic matrix of a set $S$ of vertices of $FQ_n$, and apply part~\ref{part:detSetS1} of Corollary~\ref{cor:detSetS}.
 
 Next we consider the case where $n$ is odd, but $n$ is not of the form $2^m-1$ or $2^m-3$.
 We can write $n = {2^{m-1} + q}$, where $q$ is odd and $q \le 2^{m-1} - 5$. Since we are assuming $m \ge 3$, $n \equiv q \pmod 4$.  Let $\c_1,\ldots,\c_{2^{m-1}-1}$ be the distinct, nonzero binary vectors of length $m-1$; label these so that $\c_1 = \1$ and $\c_{2t} + \c_{2t+1} = \1$ for $t \ge 1$. This implies  
\begin{equation}\label{eqn:sumci}
 \sum_{i=1}^q \c_i= \begin{cases}
 \1, \quad &\text{if } q \equiv 1 \pmod 4,\\
 \0,  &\text{if } q \equiv 3 \pmod 4.
 \end{cases}
\end{equation}

 Let $X$ be the $(m+1) \times n$ binary matrix 
 \[
X = \begin{bmatrix} 0 & 0 & \dots & 0 & 0 & 0 & 0 & \dots & 0 & 0 \\
1 & 1  & \dots & 1 & 0 & 0 & 0 & \dots & 0 & 0 \\
\c_1 & \c_2  & \dots & \c_{2^{m-1}-2} & \c_{2^{m-1}-1} & \c_1 & \c_2 & \dots & \c_q &  \c_{q+1}
\end{bmatrix}.
\]
Since $\sum_{i=1}^{2^{m-1}-1} \c_i = \0$, using Equation~\ref{eqn:sumci}, the sum of the columns of $X$ is
\[
\sum_{i=1}^n {\bf col}_i X = \begin{cases} 
\begin{bmatrix} 0 \\ 0 \\ \c_{q+2} \end{bmatrix}, \quad \text{if } q \equiv 1 \pmod 4,\\
\\
\begin{bmatrix} 0 \\ 0 \\ \c_{q+1} \end{bmatrix}, \quad \text{if } q \equiv 3 \pmod 4.\\
\end{cases}
\]

By Corollary~\ref{cor:detSetS}, $X$ is the characteristic matrix of a determining set for $FQ_n$ of size $m+1$. Thus, $\det(FQ_n)\le m+1$ and by Theorem~\ref{prop:DetFQbounds}, $\det(FQ_n)=m+1$.

The argument above will not work in the case  $n=2^m-1$ because then $q = 2^{m-1}-1$ and so there is no nonzero, binary vector $\c_{q+1}$ of length $m-1$ to put into matrix $X$. If $n=2^m-3$, then $q = 2^{m-1}-3$ and so $q+2 = 2^{m-1}-1$. This means that the sum of the columns of $X$ is itself a column of $X$ and so we cannot apply Lemma~\ref{lem:Fixins}.
 
 Finally, we show that if $n \in\{ 2^m-1, 2^m - 3\}$, where $m \ge 3$, then $\det(FQ_n)$ is $\lceil\lg n\rceil + 2$, the upper bound in Propostion~\ref{prop:DetFQbounds}. For these values of $n$, the lower and upper bounds in Proposition~\ref{prop:DetFQbounds} are  $\lceil \lg(n+1)\rceil + 1 = m+1$ and $\lceil \lg(n)\rceil + 2 = m+2$, respectively. Since these differ by 1, we need only show that $\det(FQ_n) > m+1$.

Suppose $n = 2^m-1.$ Let $S$ be a determining set for $FQ_n$ of size $m+1$ with first element  $\0$. By Corollary~\ref{cor:NecessaryXS},  the  characteristic matrix $X(S)$ has top row all 0's  and distinct, nonzero columns. Since $n = 2^m-1$, the last $m$ rows of $X(S)$ must include every possible nonzero binary $m$-vector. In particular, there must be $i \in [n]$ such that 
\[
{\bf col}_i X(S) 
= \begin{bmatrix} 0 \\ \1 \end{bmatrix}.
\]
For every nonzero binary vector $\c_j\neq \1$ of length $m$, there is a unique nonzero  $\c_k$ such that $\c_j + \c_k = \1 = \c_i$. 
Assuming that ${\bf col}_j X(S) = \begin{bmatrix} 0 \\ \c_j\end{bmatrix}$, the remaining columns of $X(S)$ can be paired up so that
\[
{\bf col}_j X(S) + {\bf col}_k X(S) = {\bf col}_i X(S).
\]
 Let $\phi\in S_{n+1}$ be defined by the transpositions $\bf e_i \leftrightarrow 1$ and $\bf e_j \leftrightarrow e_k$ whenever $\c_j + \c_k = \1 = \c_i$.  By equation (\ref{eq:phiAction}), this nontrivial  $\phi$ fixes all vertices in $S$ and so $S$ cannot be a determining set. 

Now suppose $n = 2^m-3$.
Again, by way of contradiction, suppose that $S$ is a determining set of $FQ_n$ of size $m+1$ with first element $\0$, implying that the characteristic matrix $X(S)$ has top row all 0's and distinct, nonzero columns. Let ${\c_1, \dots, \c_n} \in \mathbb Z_2^m$ denote the columns of $X(S)$ with the top row of 0's removed. Let $\c_{n+1}, \c_{n+2}$ denote the remaining two nonzero binary $m$-vectors. Then 
 \[
 \0 = \c_1 + \dots + \c_{2^m-1} = \left (\sum_{j=1}^{n} \c_j\right ) + \c_{n+1} + \c_{n+2}.
 \]
 Since $\mathbb Z_2^m$ is closed under addition, $\sum_{j=1}^{n} {\c_j} \in \mathbb Z_2^m$. We will show that $\sum_{j=1}^n {\c_j} \in \mathbb Z_2^m$  must be one of $\c_1, \dots , \c_n$ by a process of elimination. 
 \begin{itemize}
\item If $\sum_{j=1}^n \c_j = \0$, then $\c_{n+1}+ \c_{n+2} = \0$, which implies $\c_{n+1} = \c_{n+2}$, a contradiction.
\item If $\sum_{j=1}^n \c_j = \c_{n+1}$, then $2\c_{n+1}+\c_{n+2} = \0$, which implies $\c_{n+2} = \0$, a contradiction. Analogously $\sum_{j=1}^{n} \c_j = \c_{n+2}$ implies $\c_{n+1} = \0$, a contradiction.
\end{itemize}

 Thus, $\sum_{j=1}^{n} \c_j = \c_i$ for some $i \in [n]$. For each $j \ne i$, there is a unique $k\ne j$ such that $\c_k = \c_i + \c_j$. 
 In this way, addition by $\c_i$ is associated with the transposition $\c_j \leftrightarrow \c_k$. 
 In particular, since $\c_i + \c_{n+1} + \c_{n+2} = \0$, $\c_{n+1}$ and $\c_{n+2}$ are transposed, which implies that elements of $\{\c_1, \dots, \c_{i-1}, \c_{i+1} \dots, \c_n\}$ are transposed with each other.

 Let $\phi\in S_{n+1}$ be defined by $\phi(\e_i) = \1$, $\phi(\1) = \bf e_i$ and for $j \ne i$,  $\phi(\e_j) = \e_k$ iff $\c_k = \c_i + \c_j$. Then $\phi$ fixes every vertex in $S$ and so $S$ cannot be a determining set for $FQ_n$.
 \end{proof}

To find the distinguishing number of $FQ_n$, we use Lemma~\ref{lem:inducedsg}, and a determining set for $FQ_n$.

\begin{thm}\label{thm:distFQn}
The distinguishing number of the folded cube is
\[ \dist (FQ_n) = \begin{cases}
 2, & \mbox{ if }n=1,\\
 4, & \mbox{ if }n=2,\\
 5, & \mbox{ if }n=3,\\
 2, & \mbox{ if }n\geq 4.
 \end{cases}  \]
\end{thm}

\begin{proof}
The cases $n \le 3$ are discussed at the beginning of this section, so suppose that $n \geq 4$.

By Lemma~\ref{lem:inducedsg}, we can show that $FQ_n$ has distinguishing number 2 by finding a determining set $S$ for $FQ_n$ for which $FQ_n[S]$ is an asymmetric subgraph. Note that $S$ need not be a minimal determining set. Such induced subgraphs for the cases $4 \leq n \leq 8$ are displayed in Figure~\ref{fig:distFQcases}. So $\dist(FQ_n)=2$ for $4\le n\le 8$. In what follows, we generalize the construction shown for the case $n=8$.

\begin{figure}[!h]
    \centering
    \begin{tikzpicture}[scale=1.2]
    \draw[line width = 1] (0,0) -- (5,0);
    \draw[line width = 1] (2,0) -- (2,1);
    \draw[fill=black!100,line width=1] (2,1) circle (.1);
    \draw[fill=black!100,line width=1] (0,0) circle (.1);
    \draw[fill=black!100,line width=1] (1,0) circle (.1);
    \draw[fill=black!100,line width=1] (2,0) circle (.1);
    \draw[fill=black!100,line width=1] (3,0) circle (.1);
    \draw[fill=black!100,line width=1] (4,0) circle (.1);
    \draw[fill=black!100,line width=1] (5,0) circle (.1);
    \draw (0,-.3) node{$1111$};
    \draw (1,-.3) node{$0000$};
    \draw (2,-.3) node{$1000$};
    \draw (3,-.3) node{$1010$};
    \draw (4,-.3) node{$0010$};
    \draw (5,-.3) node{$0110$};
    \draw (2,1.3) node{$1100$};
\begin{scope}[shift={(7,0)}]
    \draw[line width = 1] (0,0) -- (5,0);
    \draw[line width = 1] (4,0) -- (4,1);
    \draw[line width = 1] (5,1) -- (4,1);
    \draw[fill=black!100,line width=1] (5,1) circle (.1);
    \draw[fill=black!100,line width=1] (4,1) circle (.1);
    \draw[fill=black!100,line width=1] (0,0) circle (.1);
    \draw[fill=black!100,line width=1] (1,0) circle (.1);
    \draw[fill=black!100,line width=1] (2,0) circle (.1);
    \draw[fill=black!100,line width=1] (3,0) circle (.1);
    \draw[fill=black!100,line width=1] (4,0) circle (.1);
    \draw[fill=black!100,line width=1] (5,0) circle (.1);
    \draw (0,-.3) node{$10101$};
    \draw (1,-.3) node{$10001$};
    \draw (2,-.3) node{$11001$};
    \draw (3,-.3) node{$11101$};
    \draw (4,-.3) node{$11111$};
    \draw (5,-.3) node{$11110$};
    \draw (4,1.3) node{$00000$};
    \draw (5,1.3) node{$01000$};
\end{scope}
\begin{scope}[shift={(0,-3)}]
    \draw[line width = 1] (0,0) -- (10.5,0);
    \draw[line width = 1] (7.5,0) -- (7.5,1);
    \draw[fill=black!100,line width=1] (7.5,1) circle (.1);
    \draw[fill=black!100,line width=1] (0,0) circle (.1);
    \draw[fill=black!100,line width=1] (1.5,0) circle (.1);
    \draw[fill=black!100,line width=1] (3,0) circle (.1);
    \draw[fill=black!100,line width=1] (4.5,0) circle (.1);
    \draw[fill=black!100,line width=1] (6,0) circle (.1);
    \draw[fill=black!100,line width=1] (7.5,0) circle (.1);
    \draw[fill=black!100,line width=1] (9,0) circle (.1);
    \draw[fill=black!100,line width=1] (10.5,0) circle (.1);
    \draw (0,-.3) node{$101010$};
    \draw (1.5,-.3) node{$100010$};
    \draw (3,-.3) node{$110010$};
    \draw (4.5,-.3) node{$110011$};
    \draw (6,-.3) node{$111011$};
    \draw (7.5,-.3) node{$111111$};
    \draw (9,-.3) node{$111101$};
    \draw (10.5,-.3) node{$111100$};
    \draw (7.5,1.3) node{$000000$};
\end{scope}

\begin{scope}[shift={(0,-6)}]
    \draw[line width = 1] (0,0) -- (12,0);
    \draw[line width = 1] (9,0) -- (9,1);
    \draw[line width = 1] (12,1) -- (9,1);
    \draw[fill=black!100,line width=1] (9,1) circle (.1);
    \draw[fill=black!100,line width=1] (10.5,1) circle (.1);
    \draw[fill=black!100,line width=1] (12,1) circle (.1);
    \draw[fill=black!100,line width=1] (0,0) circle (.1);
    \draw[fill=black!100,line width=1] (1.5,0) circle (.1);
    \draw[fill=black!100,line width=1] (3,0) circle (.1);
    \draw[fill=black!100,line width=1] (4.5,0) circle (.1);
    \draw[fill=black!100,line width=1] (6,0) circle (.1);
    \draw[fill=black!100,line width=1] (7.5,0) circle (.1);
    \draw[fill=black!100,line width=1] (9,0) circle (.1);
    \draw[fill=black!100,line width=1] (10.5,0) circle (.1);
    \draw[fill=black!100,line width=1] (12,0) circle (.1);
    \draw (0,-.3) node{$1010101$};
    \draw (1.5,-.3) node{$1110101$};
    \draw (3,-.3) node{$1100101$};
    \draw (4.5,-.3) node{$1100111$};
    \draw (6,-.3) node{$1100110$};
    \draw (7.5,-.3) node{$1110110$};
    \draw (9,-.3) node{$1111110$};
    \draw (10.5,-.3) node{$1111100$};
    \draw (12,-.3) node{$1111000$};
    \draw (9,1.3) node{$1111111$};
    \draw (10.5,1.3) node{$0000000$};
    \draw (12,1.3) node{$1000000$};
\end{scope}

\begin{scope}[shift={(0,-9)}]
    \draw[line width = 1] (0,0) -- (12,0);
    \draw[line width = 1] (9,0) -- (9,1);
    \draw[line width = 1] (12,1) -- (9,1);
    \draw[fill=black!100,line width=1] (9,1) circle (.1);
    \draw[fill=black!100,line width=1] (10.5,1) circle (.1);
    \draw[fill=black!100,line width=1] (12,1) circle (.1);
    \draw[fill=black!100,line width=1] (0,0) circle (.1);
    \draw[fill=black!100,line width=1] (1.5,0) circle (.1);
    \draw[fill=black!100,line width=1] (3,0) circle (.1);
    \draw[fill=black!100,line width=1] (4.5,0) circle (.1);
    \draw[fill=black!100,line width=1] (6,0) circle (.1);
    \draw[fill=black!100,line width=1] (7.5,0) circle (.1);
    \draw[fill=black!100,line width=1] (9,0) circle (.1);
    \draw[fill=black!100,line width=1] (10.5,0) circle (.1);
    \draw[fill=black!100,line width=1] (12,0) circle (.1);
    \draw (0,-.3) node{$10101010$};
    \draw (1.5,-.3) node{$11101010$};
    \draw (3,-.3) node{$11001010$};
    \draw (4.5,-.3) node{$11001110$};
    \draw (6,-.3) node{$11001100$};
    \draw (7.5,-.3) node{$11101100$};
    \draw (9,-.3) node{$11111100$};
    \draw (10.5,-.3) node{$11110100$};
    \draw (12,-.3) node{$11110000$};
    \draw (9,1.3) node{$11111110$};
    \draw (10.5,1.3) node{$11111111$};
    \draw (12,1.3) node{$00000000$};
\end{scope}

    \end{tikzpicture}
    \caption{Asymmetric induced subgraphs of $FQ_n$ that contain a determining set $S$, for $n=4,5,6,7,8$.}
    \label{fig:distFQcases}
\end{figure}
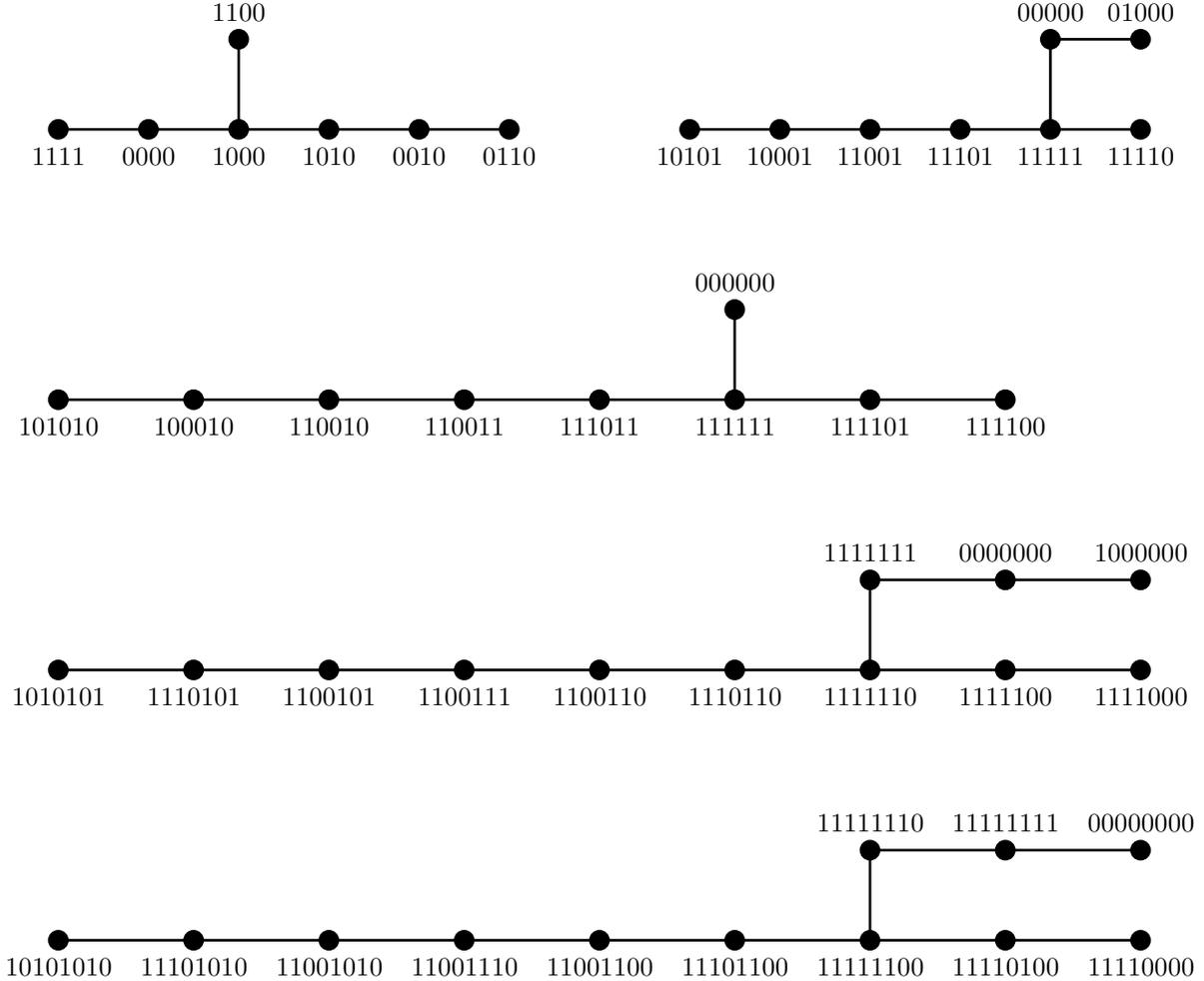

Assume $n \ge 8.$ Let $S_Q$ be a minimum determining set for $Q_n$ containing $\0$. Note that such a set does not contain $\1$. By the proof of Proposition~\ref{prop:DetFQbounds}, $S = S_Q \cup \{ \1 \}$ is a determining set for $FQ_n$ 
We will make use of the determining set $S_Q$ described in Theorem 3 in~\cite{Bo2009a}. First let $n=2^r$. For $1\leq i \leq r$, let $\V_i$ be the length $n$ vector with 1's in positions $1 \leq t2^i+s \leq n$ where $1\leq s \leq 2^{i-1}$ and $0 \leq t \leq 2^{r-i}-1$ and 0's elsewhere. In other words, $\V_i$ is a length $n=2^r$ vector that alternates $2^{i-1}$ consecutive 1's with $2^{i-1}$ consecutive 0's. If $2^{r-1}<n<2^r$, we let $\V_i$ be as described for $n=2^r$, but we discard the last $2^r-n$ positions  to obtain a vector of length $n$. Finally, we let $\V_0 = \0$. 
By~\cite{Bo2009a}, the set $S = \{\V_0,\V_1,\dots,\V_r\}$ is a determining set for $Q_n$, where $r = \lceil \lg n\rceil$.
For example, if $n=8$, then $r = 3$ and we have $\V_0=00000000$, $\V_1=10101010$, $\V_2=11001100$, and $\V_3=11110000$. If $n=10$, then $r = 4$ and we have $\V_0=0000000000$, $\V_1=1010101010$, $\V_2=1100110011$,  $\V_3=1111000011$ and $\V_4=1111111100$.

Next, we will find a superset $S'$ of $S$ such that the induced subgraph $FQ_n[S']$ is asymmetric. 
First, note that every vector in the set $S\setminus \{\V_0\} = \{\V_1,\dots,\V_r \}$ has  1 in its first position, and therefore no two of them differ in all positions. Furthermore, if $2 \le i<j$ the $\V_i$ and $\V_j$ will differ in positions $2^i+1$ and $2^i+2$, while $\V_1$ differs from $\V_2$ in positions 2 and 3, and $\V_1$ differs from $\V_j$ for $j \ge 3$ in positions 2 and 4. Thus, vertices in $S\setminus \{V_0\}$ differ pairwise in at least 2 positions 
and form an independent set. We will start by forming a $\V_1\V_{r}$-path that contains each vector in the set $\{ \V_1,\dots,\V_{r}\}$, in that order. For each pair $\V_i$ and $\V_{i+1}$, $1 \leq i \leq r-1$, let $\V_{i,k}$ be the vector that differs from $\V_i$ only in the first $k$ positions in which $\V_i$ and $\V_{i+1}$ differ, for $1 \leq k \leq h(\V_i,\V_{i+1})-1$.
In other words, we move from $\V_i$ to $\V_{i+1}$ by flipping one position at a time, in order. 
For example, in $FQ_8$, we construct the following path from $\V_1=10101010$ to $\V_2=11001100$: 
\[ 10101010\; -\; 11101010 \; -\; 11001010 \; -\; 11001110 \; -\;11001100. \]
Then, if $h(\V_i,\V_{i+1})=l_i$, the vertices $\V_{i}=\V_{i,0},\V_{i,1} ,\dots,\V_{i,l_i-1} ,\V_{i+1}$ induce a path in $FQ_n$. We can see that this is true by noting that, again, all of these vectors have a 1 in their first position, and therefore no two of them differ in all positions. 
 Hence, the only edges between them are edges in $Q_n$. If there were any edges other than those edges between consecutive vectors in the list, we would have a path of length less than $l_i$ from $\V_i$ to $\V_{i+1}$ in the hypercube $Q_n$. This is not possible, since the graph distance between two vertices in $Q_n$ is equal to their Hamming distance.

We now claim that the full set $\V_{1},\V_{1,1} ,\dots,\V_{2},\V_{2,1} ,\dots,\V_{r}$ induces a path in $FQ_n$. 
Let $m=2^{\lfloor \lg n \rfloor}$. We will show that any two non-consecutive vectors in the list differ in at least 2 positions among the first $m$ positions. For any vector $\V\in V(FQ_n)$, we let $\V^L$ indicate the vector formed by  the first block of $m/2$ positions of $\V$, and $\V^R$ the next block of $m/2$ positions. By construction, we have that $h(\V_i^L,\V_j^L)=h(\V_i^R,\V_j^R)=m/4$ for $i \ne j$. Suppose that there is an edge from $\V_{i,s}$ to $\V_{j,t}$, with $1\leq i < j\leq r$ and $0\leq s < l_i$, $0< t \leq l_j$, with notation $\V_{i,0}=\V_i$. We consider the following cases.

\begin{itemize}
\item[(i)] Suppose that $s,t \leq m/4$ or that $s,t \geq m/4$. In the former case, we have that $\V_{i,s}^R=\V_i^R$ and $\V_{j,t}^R=\V_j^R$. Therefore, $h(\V_{i,s},\V_{j,t})\geq m/4\geq 2$. In the latter case, we have $\V_{i,s}^L=\V_{i+1}^L$ and $\V_{j,t}^L=\V_{j+1}^L$. Therefore, $h(\V_{i,s},\V_{j,t})\geq m/4\geq 2$, and we have arrived at a contradiction.
\item[(ii)] Suppose that $s>m/4$, $t< m/4$ and $i=j-1$. Then, $\V_{i,s}$ differs from $\V_j$ in at least one position in the set $m/2+1,\dots,n$, while $\V_{j,t}$ does not. 
Furthermore, $\V_{j,t}$ differs from $\V_j$ in at least one position in the set $1,\dots,m/2$, while $\V_{i,s}$ does not. Therefore, $h(\V_{i,s},\V_{j,t})\geq 2$, and we have arrived at a contradiction.
\item[(iii)] Suppose that $s>m/4$, $t<m/4$ and $i<j-1$. Then $\V_{j,t}^R=\V_j^R$. If $s\geq m/2$, then $\V_{i,s}^R=\V_{i+1}^R$ and $\V_{i,s}^L=\V_{i+1}^L$. In particular, since $h(\V_{i+1}^R,\V_j^R)=m/4$, we must have $h(\V_{i,s},\V_{j,t})\geq m/4\geq 2$. If $s<m/2$, we have a path $\V_{i,m/2},\V_{i,m/2-1},\dots,\V_{i,s},\V_{j,t},\V_{j,t-1},\dots,\V_j$. This path has length $m/2-s+1+t<m/2$, and uses only edges from $Q_n$. This contradicts the fact that \[
h(\V_{i,m/2},\V_j) \ge h(\V_{i, m/2}^R, \V_j^R) + h(\V_{i, m/2}^L, \V_j^L)= h(\V_{i+1}^R, \V_j^R) + h(\V_{i+1}^L, \V_j^L) = m/2.
\] 
Therefore, $h(\V_{i,s},\V_{j,t})\geq 2$, and we have arrived at a contradiction.
\item[(iv)] Suppose that $s <m/4$ and $t> m/4$. Then, we have a path $\V_{i},\dots,\V_{i,s},\V_{j,t},\dots,\V_{j,\max(m/2,t)}$. This path has length $s+1+(m/2-t)<m/2$ if $t\leq m/2$ and length $s+1<m/2$ otherwise,
and uses only edges from $Q_n$. This contradicts the fact that $h(\V_{i},\V_{j,m/2})\geq m/2$. Therefore, $h(\V_{i,s},\V_{j,t})\geq 2$, and we have arrived at a contradiction.
\end{itemize}

We label this $\V_1 \V_{r}$ path $P$. We now extend $P$ to a tree that also includes $\V_0$. Note that the vector $\V_{r-1,m/2}=\V_C$ has minimum Hamming distance to $\1$ over all vectors on the path. 
This is easy to see for the vectors on the subpath from $\V_{r-1}$ to $\V_{r}$, since this is where we switch from flipping 0's to 1's to flipping 1's to 0's. 

Furthermore, note that $\V_C$ has all 1's in positions $1,\dots,\min (n,3m/2)$. If $m<n\leq 3m/2$, then $\V_C=\1$. If $n> 3m/2$, then $\V_C$ has at most $n/4$ 0's. 
Consider the three blocks of positions: $1,\dots,m/2$; $m/2+1,\dots,m$; and $m+1,\dots ,3m/2$. For any vector on the path that precedes $\V_{r-1}$, at least two out of those three blocks have the same number of 1's and 0's, and therefore the vector has 0's in at least $m/2>n/4$ places. If $\V_C\neq\1$, we add the path $\V_C, \dots, \1$ to our tree, by flipping one position at a time as before. All of these vertices have at least as many 1's as $\V_C$ and therefore do not have edges to $P$. We label our new graph $T_1$.

Let $T = T_1+\V_0$. Note that the only vertex in $T_1$ that $\V_0$ is adjacent to is $\1$. Then $T$ is isomorphic to a subdivision of $K_{1,3}$, and is therefore asymmetric if and only if its three leaves have distinct distances from the unique vertex of degree 3, $\V_C$. If this is the case, then we have found a superset $S'$ of the determining set $S$, such that the subgraph of $FQ_n$ induced by $S'$ is asymmetric. 
By Lemma~\ref{lem:inducedsg}, this is enough to prove that $\dist (FQ_n)=2$ and we are done. 
Suppose that $T$ is not asymmetric. We have that $d (\V_1,\V_C)>m/2$ in $T$, while $d (\V_0,\V_C),d (\V_{\lceil \lg n \rceil+1},\V_C)<m/2$, so it must be the case that $d (\V_0,\V_C)= d (\V_{\lceil \lg n \rceil+1},\V_C)$. We add the vector $100\dots 0$, with an edge to $\V_0$ to $T$ to create $T'$. It is easy to see that $100\dots 0$ does not have edges in $FQ_n$ to any other vertices in $T$. Therefore, $T'$ is an asymmetric induced subgraph of $FQ_n$ that contains all the vertices of $S$, and we are done by Lemma~\ref{lem:inducedsg}.
\end{proof}

\begin{cor}\label{cor:costfolded}
For $n\geq 4$, we have \[ \rho(FQ_n)=O(n \lg n). \]
\end{cor}

\begin{proof}
 By the construction in the proof of Theorem~\ref{thm:distFQn}, the $\V_1\V_r$ path $P$ has length at most $1+(r-1)\cdot n/2$, with at most $n/4+2$ vectors added to create the asymmetric tree. The set that induces this tree forms one of the two color classes in a 2-distinguishing coloring.
\end{proof}

\section{Enhanced Hypercubes}\label{sec:enhanced}

In~\cite{TzWe1991}, Tzeng and Wei introduced the enhanced hypercube. The \emph{enhanced hypercube} $Q_{n,k}$ for $1\le k \le n-1$ is the result of adding edges between vertices in $Q_n$ that differ in their rightmost $n-k+1$ positions. That is, we add the edge between $\mathbf x = x_1\ldots x_n$ and $\mathbf y = y_1 \ldots y_n$ if $y_i = x_i$ for $1 \le i\le k-1$ and $y_i = 1+x_i \pmod{2}$ for $k \le i \le n$. 

Note that $Q_{n,1} = FQ_n$. Generalizing this idea, Lu and Huang~\cite{LuHu2021} show that $Q_{n,k}=Q_{k-1} \Box FQ_{n-k+1}$. Additionally, for any $k$, we have that $Q_{n,k}$ is a subgraph of $AQ_n$. We can use this with several results from Sabidussi~\cite{Sa1959}, Boutin~\cite{Bo2009a}, and Imrich and Klav\v{z}ar~\cite{ImKl2006} to find determining and distinguishing numbers for $Q_{n,k}$.

 \begin{prop}\label{prop:cartesianautomorphismgroup}
\cite{Sa1959} Let $G_1,\ldots, G_m$ be connected graphs which are relatively prime with respect to the Cartesian product. 
 Then $\aut(G_1\Box \cdots \Box G_m) \cong \aut(G_1) \times \cdots \times \aut(G_m)$. 
 \end{prop}

\begin{thm}\label{thm:cartproddet}
\cite{Bo2009a} Let $G$ be a connected graph and $G = G_1^{k_1} \Box \cdots \Box G_m^{k_m}$ where $G_i$ are prime with respect to the Cartesian product.
Then $\det(G) = \max\{\det(G_i)^{k_i}: 1\le i \le m\}$. 
\end{thm}

 \begin{prop}\label{prop:cartesiandistlemma}
\cite{ImKl2006} Let $G$ and $H$ be connected graphs that are relatively prime with respect to the Cartesian product.
 If $\dist(G) = 2$ and $2 \le \dist(H) \le 3$ then $\dist(G\Box H) = 2$. 
 \end{prop}
 
  \begin{prop}\label{prop:cartesiandistlemma2}
\cite{ImKl2006} Let $G,H$ be connected graphs with $3 \le |G| \le |H|+1$. If $G$ and $H$ are relatively 
 prime with respect to the Cartesian product, then $\dist(G \Box H) \le \max\{2,\dist(H)\}$. 
 \end{prop}

With these we can state our result for the determining number of $Q_{n,k}$. With Proposition~\ref{prop:Qnkdet}, Theorem~\ref{thm:DetQn} and Theorem~\ref{thm:DetFQn}, we can identify the determining number of $Q_{n,k}$ for any $n$ and $k$. 

\begin{prop}\label{prop:Qnkdet}
For $n\ge 2$ and $1 \le k \le n-1$, $\det(Q_{n,k})=\max\{ \det(Q_{k-1}), \det(FQ_{n-k+1})\}$.
\end{prop}
\begin{proof}
Lu and Huang~\cite{LuHu2021} proved that if $\ell \ge 2$, then $FQ_\ell$ has no Cartesian factor of $K_2$. Thus, the decomposition of $Q_{n,k}$ as $Q_{k-1} \Box FQ_{n-k+1}$ is into relatively prime components 
and the result follows from Theorem~\ref{thm:cartproddet}.
\end{proof}

For $2 \le n \le 5$, we computational compute the distinguishing numbers of the enhanced hypercubes are summarized in Table~\ref{tab:distQnk}.

\begin{table}
    \centering
    \begin{tabular}{|c|c|c|c|c|} \hline
         $_k \backslash ^n$ & 2 & 3 & 4 & 5 \\ \hline
         1 & 4 & 5 & 2 & 2 \\ \hline
         2 & $\cdot$ &  2 & 2  & 2  \\ \hline
         3 & $\cdot$ & $\cdot$ & 2  & 2  \\\hline
         4 & $\cdot$ & $\cdot$ & $\cdot$ & 2  \\ \hline
    \end{tabular}
    \caption{The distinguishing numbers of the enhanced hypercube $Q_{n,k}$ for $2 \le n \le 5$ and $1 \le k \le n-1$.}
    \label{tab:distQnk}
\end{table}

\begin{thm}\label{thm:distenhanced}
 For $n \ge 4$ and $1 \le k \le n-1$, $\dist(Q_{n,k}) = 2$. 
\end{thm}
\begin{proof}
We verify the cases for $4\le n \le 5$ by computer. See Table~\ref{tab:distQnk}. Suppose that $n \ge 6$.

 We have $Q_{n,k} = Q_{k-1} \Box FQ_{n-k+1}$ is a decomposition into components that are relatively prime with respect to the Cartesian product. 
 We consider cases based on $k$. 
 
 For $k = 1$, we have $Q_{n,1} = FQ_n$, and the result follows directly from Theorem~\ref{thm:distFQn}. 
 
 For $k = 2$, we have $Q_{n,2} = K_2 \Box FQ_{n-1}$. Consider the coloring where one copy of $FQ_{n-1}$ has a 2-distinguishing coloring as given by Theorem~\ref{thm:distFQn}, and the other copy has a single color. Using the automorphism group given by Proposition~\ref{prop:cartesianautomorphismgroup}, it follows that this coloring is distinguishing. On the other hand, two colors are needed since otherwise the two copies of $FQ_{n-1}$ could be exchanged.

 For $k \ge 3$, we have $2 \le \dist(Q_{k-1}) \le 3$ by Theorem~\ref{thm:distQn}. For $n \ge k+3$, we have $\dist(FQ_{n-k+1}) = 2$ by Theorem~\ref{thm:distFQn}. Thus, for $3 \le k \le n-3$ result follows from Proposition~\ref{prop:cartesiandistlemma}.
 
 For $Q_{6,4}$, we have computationally verified that $\dist(Q_{6,4})=2$. Otherwise, for $k = n-2$ and $n-1$ with $n\ge 7$ and $n \ge 6$ respectively, letting the hypercube and folded hypercube components of the decomposition be $H$ and $G$ respectively in Proposition~\ref{prop:cartesiandistlemma2} gives $\dist(Q_{n,n-1}) \le 2$. Since the hypercube factor gives a nontrivial automorphism, equality holds.  
\end{proof}

\section{Augmented Hypercubes}\label{sec:augmented}

Introduced by Choudum and Sunitha in~\cite{ChSu2002}, augmented hypercubes are supergraphs of traditional hypercubes; the extra edges connect opposite vertices at each stage of an iterative construction.  
Let $AQ_1 = K_2$. To construct $AQ_n$ for $n>1$, start with two copies of the augmented cube $AQ_{n-1}$. Let $V(AQ_{n-1}^0)$ and $V(AQ_{n-1}^1)$ denote the set of binary $n$-vectors whose first position is  $0$ or  $1$ respectively. Add an edge between  $0a_2 \dots a_n$ and $1 b_2 \dots b_n$  if and only if either $a_i = b_i$ for all $2 \leq i \leq n$ (such edges would be in the traditional hypercube) or $a_i \neq b_i$ for all $2 \leq i \leq n$ (such edges are the augmented edges). Equivalently,  $\mbf a = a_1 a_2 \dots a_n$ and $\mbf b = b_1 b_2 \dots b_n$ are adjacent if and only if either these two vectors differ in exactly one position, or for some $0 \leq \ell \leq n-2$, the first $\ell$ positions are the same, and the remaining $n - \ell$ positions are all different.  In particular, $AQ_n$ is a supergraph of $FQ_n.$

 Note that $AQ_n$ is a $(2n-1)$-regular graph of order $2^n$.  In~\cite{ChSu2002}, Choudum and Sunitha prove that $AQ_n$ is vertex-transitive.  We can use their approach to prove a somewhat stronger result in Lemma~\ref{lem:switcheroo}. However, we will show later, $AQ_n$ is neither edge-transitive nor arc-transitive for $n \ge 3$.
 
 \begin{lemma}\label{lem:switcheroo} 
 Let $n\ge 2$. For all $\mbf a, \mbf b \in V(AQ_n)$, there is an automorphism $\rho \in \aut(AQ_n)$ such that $\rho(\mbf a)= \mbf b$ and $\rho(\mbf b) = \mbf a$. In particular, $AQ_n$ is vertex-transitive.
 \end{lemma}

\begin{proof}
 It is easy to verify that translations are automorphisms of $AQ_n$. If we let ${\mbf c} = {\mbf a} + {\mbf b}$, then
 $
 \rho_{\mbf c}(\mbf a) = \mbf a + (\mbf a+ \mbf b) = \mbf 0 + \mbf b = \mbf b
 $
 and similarly $\rho_{\mbf c}(\mbf b) = \mbf a.$
\end{proof}

In~\cite{ChSu2008}, Choudum and Sunitha find all automorphisms of $AQ_n$. We combine a couple of their results to bound the determining number.

\begin{thm}\label{thm:detAQ1}  
For all $n\geq 4$, $\det(AQ_n)\leq 3$.
\end{thm}

\begin{proof}
Let $\mbf x \in V(AQ_n)$. Lemma 3.2 in~\cite{ChSu2008} asserts that if $\mbf x_1$  is the vertex of $AQ_n$ that differs from $\mbf x$ only in the first position and $\mbf x_{n}$ is the vertex that differs only in the  $n^{th}$ position, then any automorphism of $AQ_n$ that fixes $\mbf x, \mbf x_1$ and $\mbf x_n$ must fix every vertex in $N[\mbf x]$. Further, Theorem 3.3 in~\cite{ChSu2008} asserts that if $\varphi, \psi \in \aut(AQ_n)$ satisfy $\varphi(\mbf a) = \psi(\mbf a)$ for all $\mbf a \in N[\mbf x]$, then $\varphi=\psi$ on $AQ_n$. By definition, this means that every closed neighborhood is a determining set. Thus, $\{\mbf x, \mbf x_1, \mbf x_n\}$ is a determining set for $AQ_n$, and $\det(AQ_n)\leq 3$.
\end{proof}

We sharpen Theorem~\ref{thm:detAQ1} in the following.
 
  \begin{thm}\label{thm:DetAQ2} 
  The determining number of the augmented cube is
  \[
 \det(AQ_n) = \begin{cases}
 1, \quad &\mbox{if } n=1,\\
 3, &\mbox{if } n=2,\\
 4, &\mbox{if } n=3,\\
 3, &\mbox{if } n = 4 \text{ or }5,\\
 2, &\mbox{if } n \geq 6.
 \end{cases}
 \]
\end{thm}
 
  \begin{proof}
 
Since $AQ_1 = K_2$ and $AQ_2 = K_4$ are complete graphs, the smallest determining set in each case is the complement of a single vertex.  Thus, $\det(AQ_1) =1$ and $\det(AQ_2) = 3$.  The complement of $AQ_3$ is two disjoint copies of $C_4$ (one cycle is induced by $000$, $101$, $110$ and $011$, and the other by the remaining four vertices); each copy of $C_4$ has determining number 2, so $\det(AQ_3) = 4$.

Before tackling $n\geq 4$, let us look at a table provided by Choudum and Sunitha  in~\cite{ChSu2008} showing the action of all automorphisms of $AQ_n$ that fix $\mbf 0$, replicated in Table~\ref{tab:AutoAQ}.  This table shows that $\0$, and by vertex transitivity any vertex of $AQ_n$, has exactly 8 automorphisms of $AQ_n$ that fix it. In particular, this implies that $\det(AQ_n) > 1$. In this table, $A$ denotes an $(n-3)$-vector, $A^R$ denotes the vector achieved by reversing the order of the positions, $\overline{A}$ denotes the vector in which each position of $A$ is replaced with its opposite in $\Z_2$. Note that $\varphi_1$ is the trivial automorphism in this table.

  \begin{table}
  \caption{Action of Automorphisms that fix $\mbf 0$}
  \[
\begin{array}{c|cccccccc|}
& \varphi_1 & \varphi_2 & \varphi_3 & \varphi_4 &\varphi_5 & \varphi_6 & \varphi_7 & \varphi_8 \\
\hline
0A00 & 0A00 & 0A00 & 0A00 & 0A00 & 0A^R00 & 0A^R00 & 0A^R00 & 0A^R00 \\
0A01 & 0A01 & 0A01 & 0A10 & 0A10 & 1\overline{A^R}11 & 1\overline{A^R}11 & 1A^R00 & 1A^R00 \\
0A10 & 0A10 & 0A10 & 0A01 & 0A01 & 1A^R00 & 1A^R00 & 1\overline{A^R}11 & 1\overline{A^R}11\\
0A11 & 0A11 & 0A11 & 0A11 & 0A11 & 0\overline{A^R}11 & 0\overline{A^R}11  & 0\overline{A^R}11  & 0\overline{A^R}11 \\
1A00 &1A00 & 1\overline{A}11 &1A00 & 1\overline{A}11 & 0A^R10 & 0A^R01 & 0A^R10 & 0A^R01 \\
1A01 &1A01 & 1\overline{A}10 &1A10 & 1\overline{A}01 & 1\overline{A^R}01 & 1\overline{A^R}10 & 1A^R10 & 1A^R01 \\
1A10 & 1A10 & 1\overline{A}10 & 1A01 & 1\overline{A}10 & 1A^R10 & 1A^R01 & 1\overline{A^R}01 & 1\overline{A^R}10\\
1A11 & 1A11 & 1\overline{A}00 & 1A11 & 1\overline{A}00 & 0\overline{A^R}01 &  0\overline{A^R}10 & 0\overline{A^R}01 &  0\overline{A^R}10\\
\hline
\end{array}
\]\label{tab:AutoAQ}
\end{table}

Now assume $n\in\{4,5\}$.  By Theorem~\ref{thm:detAQ1}, to show that $\det(AQ_4) = \det(AQ_5) = 3$, it suffices to show that for any two distinct vertices $\mbf x$ and $\mbf y$, there exists a nontrivial automorphism fixing both. Since $AQ_n$ is vertex-transitive, we may assume $\mbf x = \mbf 0$. Notice that if $n=4$ then $A \in \{0,1\}$ and so $A=A^R$. Further, if $n=5$ then $A\in \{00,01,10,11\}$.  For $A\in \{00,11\}$ then $A=A^R$, while for $A\in\{01,10\}$ then $A=\overline{A^R}$. Thus, for $n=4$ and $n=5$ either $A=A^R$ or $A=\overline{A^R}$.

\begin{itemize} 
\item Suppose $\mbf y$ is of the form $1A01$.  If $A = A^R$, then $\varphi_8$ fixes both, and if $A = \overline{A^R}$, then $\varphi_5$ fixes $\0$ and $\y$.
 
\item Suppose $\mbf y$ is of the form $1A10$. If $A = A^R$, then $\varphi_5$ fixes both, and if $A = \overline{A^R}$, then $\varphi_8$ fixes both.
  
\item If $\mbf y$ is of the form $0A00$ $0A01$, $0A10$ or $0A11$,  then the nontrivial automorphism $\varphi_2$ fixes both $\0$ and $\mbf y$.  
      
\item If $\mbf y$ is of the form $1A00$ or $1A11$, then the nontrivial automorphism $\varphi_3$ fixes both.
\end{itemize}
  
Finally, assume $n \geq 6$. We must find a determining set consisting of two vertices. We let one of these be $\mbf 0$ and let the other be $\mbf y = 1A01$, where $A = 111 \dots 10$ (that is, $A$ consists of $n-4$ 1's followed by a single 0). Then
  \[
  \overline A = 000 \dots 01, \quad A^R = 011\dots 11, \quad \overline{A^R} = 100 \dots 00.
  \]
  In particular, $A$, $\overline A$, $A^R$ and $\overline{A^R}$ are distinct.  From Table~\ref{tab:AutoAQ}, the images of $\mbf y = 1A01$ under $\varphi_1, \varphi_2, \dots , \varphi_8$, the automorphisms that fix $\mbf 0$, are in order, the eight distinct vertices
  \[\mbf y = 1A01, \,
  1\overline A 10, \,
   1 A 10,\,
   1\overline A 01,\,
  1\overline{A^R}01,\,
   1\overline{A^R}10,\,
   1A^R 10,\,
   1A^R 01.
  \]
  Thus, the only automorphism that fixes both $\mbf 0$ and $\mbf y$ is the identity, and therefore $D= \{\mbf 0, \mbf y\}$ is a determining set.
  \end{proof}

 \begin{prop}\label{prop:AugmentedTransitive}
 For $n \leq 2$, $AQ_n$ is arc-transitive, but for $n \geq 3$, $AQ_n$ is neither arc-transitive nor edge-transitive.
 \end{prop}
  \begin{proof}
  As noted in the proof of Theorem~\ref{thm:DetAQ2}, $AQ_1 = K_2$ and $AQ_2 = K_4$; complete graphs are clearly arc-transitive. Recall that arc-transitivity implies edge-transitivity.
 
  Now assume $n \geq 4$. Let $\mbf a = 1A00$ where $A$ consists of $(n-3)$ 0's, and let $\mbf b = 0B11$ where $B$ consists of  $(n-3)$ 1's. Note that $\mbf a$ differs from $\0$ in exactly 1 position, and $\mbf b$ is the same as $\0$ for the first position then differs from $\0$ in all remaining positions. Thus, both are neighbors of $\0$. Any automorphism taking the arc $(\mbf 0, \mbf a)$ to the arc $(\mbf 0, \mbf b)$ would have to be an automorphism fixing $\mbf 0$. However, from the table in the proof of Theorem~\ref{thm:DetAQ2}, no such automorphism $\varphi$ satisfies $\varphi(\mbf a) = \mbf b$. 
 
  In fact, no automorphism can take the edge $\{\0, \mbf a\}$ to the edge $\{ \0, \mbf b\}$. From the above, no such  automorphism  can fix  $\mbf 0$.  Suppose  $\psi\in \aut(AQ_n)$ satisfies $\psi(\mbf 0) = \mbf b$ and $\psi(\mbf a) = \mbf 0$.  Let $\rho_{\b} \in \aut(AQ_n)$ be translation by $\mbf b$.
  Then $\rho_\b \circ \psi$ is an automorphism taking the arc $(\mbf 0, \mbf a)$ to the arc $(\mbf 0, \mbf b)$, which was proved  impossible.
 
Finally assume $n = 3$. In Lemma 3.2(1) of ~\cite{ChSu2008}, Choudum and Sunitha prove that every automorphism of $AQ_3$ that fixes $\mbf 0$ also fixes 011.  Hence, no automorphism takes the arc $(\mbf 0, 011)$ to $(\mbf 0, 100)$.
\end{proof}

\begin{thm} \cite{Cha2008}
The distinguishing number of the augmented cube is 
 \[
 \dist(AQ_n) = \begin{cases}
 2, \quad &\mbox{ if } n=1,\\
 4, &\mbox{ if } n=2,\\
 3, &\mbox{ if } n=3,\\
 2, &\mbox{ if } n \geq 4.
 \end{cases}
 \] 

\end{thm}

  For $n \geq 4$,  since $\dist(AQ_n) = 2$ we can consider the cost of 2-distinguishing.

\begin{thm}\label{thm:costAQ}
For all $n\geq 4$, $\rho(AQ_n) = 3$.
\end{thm}

\begin{proof}
Consider a $2$-coloring of $AQ_n$ where arbitrary vertices $\mbf a$ and $\mbf b$ are colored blue and the rest are colored red. 
By Lemma~\ref{lem:switcheroo}, there exists an automorphism that interchanges $\mbf a$ and $\mbf b$; this nontrivial automorphism preserves the color classes. Hence, this 2-coloring is not distinguishing. This shows  $\rho(AQ_n)>2.$

To show that it is possible to have a 2-distinguishing coloring with one color class of size 3, let $C = \{\mbf 0, 1B1, 0\overline B 0\}$, where $B$ is a string of $n-2$ 0's. Let $\psi$ be an automorphism of $AQ_n$ that preserves $C$ setwise. Note that $1B1$ and $0\overline B 0 $ are adjacent because they differ in all bits, but $\mbf 0$ is not adjacent to either of these two. Hence, $\psi$ must fix $\mbf 0$ and so must be one of $\varphi_1, \dots, \varphi_8$ from Table~\ref{tab:AutoAQ}. From the sixth row of the table, $\psi(1B1)$ must be a string beginning with $1$, which means $\psi(1B1) \neq 0\overline B 0$.  Hence, $\psi$ must fix $C$ pointwise; $\varphi_8$ is the only nontrivial automorphism that fixes both $\mbf 0$ and $1B1$, but $\varphi_8(0 \overline B 0) \neq 0 \overline B 0$. Thus, $\psi$ must be the identity.
 \end{proof}

\section{Locally Twisted Hypercubes}\label{sec:locallytwisted}

In~\cite{YaMeEv2005}, Yang, Megson, and Evans introduce the locally twisted hypercube, $LTQ_n$. We give a recursive definition here. The authors also give a non-recursive definition and show that $LTQ_n$ is isomorphic to a graph consisting of two copies of $Q_{n-1}$ with a perfect matching between them.

Let $LTQ_2=Q_2$. For $n>2$,  to construct $LTQ_n$, start with 2 disjoint copies of $LTQ_{n-1}$. Then prefix the set of the vertices in the first copy with a $0$ and the vertices in the second copy with a $1$. Finally, for all vertices $0x_2x_3\dots x_n$ in the first copy, add an edge to $1(x_2+x_n)x_3\dots x_n$ where addition is modulo $2$.

\begin{prop}\label{prop:LTQ}
  For $n\ge 4$, $\det(LTQ_n) = 1$, $\dist(LTQ_n)=2$, and $\rho(LTQ_n)=1$. Furthermore, $\det(LTQ_3)=\dist(LTQ_3)=2$ and $\rho(LTQ_3)=3$. 
\end{prop}
\begin{proof}
 By~\cite{ChMaYa2021}  for $n\geq 4$, $\aut(LTQ_n)=\Z_2^{n-1}$ where each automorphism acts by translation. More precisely, each automorphism acts by adding an element of $\Z_{2}^{n-1}$ to the first $n-1$ bits of each vertex in $LTQ_n$. Thus, the only automorphism that fixes $\bf{0}$ is trivial and $\det(LTQ_n) = 1$, $\dist(LTQ_n)=2$, and $\rho(LTQ_n)=1$. By~\cite{ChMaYa2021}, when $n=3$, $\aut(LTQ_3)=D_{16}$, the automorphisms of the octagon. We conclude $\det(LTQ_3)=\dist(LTQ_3)=2$ and $\rho(LTQ_3)=3$. 
\end{proof}

\section{Open Questions}

Theorem~\ref{thm:distdetpowers} shows that for $n\ge 4$, we have $\det(Q_n^2) \le n$ and $\rho(Q_n^2) \le n+1$. From~\cite{MiPe1994}, this extends to all even powers of hypercubes. We ask the following: 

\begin{quest} Is there a better upper bound, or can we find an exact value, for $\det(Q_n^2)$? 
\end{quest}

\begin{quest} Is there a better upper bound, or can we find an exact value, for $\rho(Q_n^2)$?
\end{quest}

For the determining number, from computation, we have that $\det(Q_4^2)=4$, $\det(Q_5^2)=5$, and $\det(Q_6^2)=4$. Thus, the upper bound in Theorem~\ref{thm:distdetpowers} is sharp, but does not always hold. 

Theorem~\ref{thm:DistHam} characterizes the $n$ and $m$ for which $H(m,n)$ is 2-distinguishable. For $2$-distinguishable $H(m,n)$, Theorem~\ref{thm:CostHam} gives that the cost is one of two values when $2 \le m -1 \le n$. However, we ask what happens when this inequality does not hold: 

\begin{quest} What is $\rho(H(m,n))$ when $\dist(H(m,n))=2$, but $n<m-1$?
\end{quest}
We note that by Theorem~\ref{thm:DistHam}, this must occur for $m \ge 4$ and $n \ge 2$. 

Finally, Theorem~\ref{thm:distFQn} shows that $\dist(FQ_n) = 2$ when $n \ge 4$ and Theorem~\ref{thm:distenhanced} shows that $\dist(Q_{n,k}) = 2$ when $n \ge 6$ and $1 \le k \le n-1$. While Corollary~\ref{cor:costfolded} shows that $\rho(FQ_n) = O(n \lg n)$, we ask: 

\begin{quest}
Is there a better upper bound for $\rho(FQ_n)$? 
\end{quest}

\begin{quest}
What are bounds for $\rho(Q_{n,k})$?
\end{quest}

\section*{Acknowledgments}

The work in this article is a result of a collaboration made possible by the Institute for Mathematics and its Applications' Workshop for Women in Graph Theory and Applications, August 2019. We also thank K.\ E.\ Perry for helpful discussions.

\end{document}